\documentclass[twoside,11pt]{amsart}
  \usepackage{mabliautoref}
  \usepackage{amsmath,amsfonts,amssymb,amsthm,mathtools,multirow,tabls,mathrsfs}
  \usepackage[a4paper,margin=1.1in]{geometry}  
  \usepackage{lmodern}
  \usepackage[all]{xy}
  \usepackage[utf8]{inputenc}
  \usepackage[T1]{fontenc}
  \usepackage[shortlabels]{enumitem}
  \usepackage{stmaryrd}

  \numberwithin{equation}{section}
  
  \input{includeNice3}
  
  \usepackage[most]{tcolorbox}
  \usetikzlibrary{shapes.geometric, arrows, positioning, fit, calc,decorations.pathmorphing,shapes}
\tikzstyle{e} = [ellipse, minimum width=1cm, minimum height=0.5cm,text centered, draw=black]
\tikzstyle{arrow} = [thick,-,>=stealth]
\tikzstyle{connection}=[inner sep=0,outer sep=0]

\tikzset{
    vertl/.style={anchor=south, rotate=90, inner sep=.4mm, pos=0.4}
}

\tikzset{
    vertr/.style={anchor=north, rotate=90, inner sep=.7mm, pos=0.4}
}

\tikzset{
    diag/.style={anchor=south, rotate=30, inner sep=.5mm}
}

\tikzset{
    ddiag/.style={anchor=south, rotate=-16, inner sep=0mm}
}

    \author[J. Baudin]{Jefferson Baudin}
  \address{\'Ecole Polytechnique F\'ed\'erale de Lausanne, Chair of Algebraic Geometry \newline 
    \indent MA C3 575 (Bâtiment MA), Station 8, CH-1015 Lausanne}
  \email{jefferson.baudin@epfl.ch}

  \author[Zs. Patakfalvi]{Zsolt Patakfalvi}
  \address{\'Ecole Polytechnique F\'ed\'erale de Lausanne, Chair of Algebraic Geometry \newline 
    \indent MA C3 635 (Bâtiment MA), Station 8, CH-1015 Lausanne}
  \email{zsolt.patakfalvi@epfl.ch}

   \author[L. Rösler]{Linus Rösler}
  \address{\'Ecole Polytechnique F\'ed\'erale de Lausanne, Chair of Algebraic Geometry \newline 
    \indent MA C3 615 (Bâtiment MA), Station 8, CH-1015 Lausanne}
  \email{linus.rosler@epfl.ch}

  \author[M. Zdanowicz]{Maciej Zdanowicz}
  \email{m.zdanowicz@gmail.com}

  \title[On Gorenstein $\bQ_p$--rational threefold and fourfold singularities]{On Gorenstein $\bQ_p$--rational threefold and fourfold singularities}
  

    \AtBeginDocument{%
   \def\MR#1{}
    }

\begin{document}

\begin{abstract}
    We prove that for $n \leq 4$ and $p > 5$, quasi--Gorenstein $F$--pure and $\bQ_p$--rational $n$--fold singularities are canonical. This is analogous to the usual fact that rational Gorenstein singularities are canonical. The proof is based on a careful analysis of the dual complex of a dlt modification of a log canonical singularity. The result for $n = 4$ is contingent upon the existence of log resolutions.
\end{abstract} 

\maketitle


\tableofcontents

\section{Introduction}

Recall that a normal variety $X$ over a field $k$ is said to have \emph{rational singularities} if it is Cohen--Macaulay and it admits a resolution of singularities $\pi \colon Y \to X$ such that for all $i > 0$, $R^i\pi_*\cO_Y = 0$.

In characteristic zero, an important theorem of \cite{Elkik_Rationalite_des_singularites_canoniques} says that Kawamata log terminal (klt, for short) singularities are rational, which in turn has many applications. The converse of this statement fails, as shown by taking a cone over an Enriques surface. Nevertheless, there is a well--known partial converse: rational and Gorenstein normal singularities are canonical (see  \autoref{main_thm:classical_statement}). In fact, this holds in any characteristic.

An important issue is that in positive characteristic, klt singularities need not be rational, or even Cohen--Macaulay (see e.g. \cite{Yasuda_Discrepancies_of_p-cyclic_quotient_varieties, Hacon_Witaszek_On_the_rationality_of_klt_sings_in_pos_char, Totaro_The_failure_of_Kodaira_vanishing_for_Fano_varieties__and_terminal___singularities_that_are_not_Cohen-Macaulay, Arvidsson_Bernasconi_Lacini_On_the_Kawamata-Viehweg_vanishing_theorem_for_log_del_Pezzo_surfaces_in_positive_characteristic, Totaro_Terminal_3folds_that_are_not_CM, Posva_Pathological_MMP_singularities_as_alpha_p_quotients}). Luckily, there exist natural weakenings of the notion of rational singularities: $W\cO$--rational (\cite[Definition 4.4.4]{Rulling_Chatzistamatiou_Hodge_Witt_and_Witt_rational}), $\bQ_p$--rational (\cite[Section 3.4]{Patakfalvi_Zdanowicz_Ordinary}) or $\bF_p$--rational (\cite[Definition 3.4]{Baudin_Bernasconi_Kawakami_Frobenius_GR_fails}) singularities. The idea is that the role of $\cO_Y$ in the definition of rational singularities is replaced by $W\cO_{Y, \bQ}$, $\bQ_{p, Y}$ or $\bF_{p, Y}$. Interestingly, these weaker notions are known to hold for klt singularities up to dimension $4$ (\cite{Gongyo_Nakamura_Tanaka_Rational_points_on_log_Fano_threefolds_over_a_finite_field,Hacon_Witaszek_On_the_relative_MMP_for_4folds_in_pos_and_mixed_char, Baudin_Bernasconi_Kawakami_Frobenius_GR_fails}). 

This led us to search for partial converses of the latter implications. Note that since klt singularities may not be Cohen--Macaulay, it seems excessive to impose Cohen--Macaulayness assumptions for this converse (see \autoref{cor:main_thm_under_F_p-CM}). Hence, we relax the Gorenstein hypothesis by imposing that our variety $X$ is quasi--Gorenstein (i.e.\ the Weil divisor $K_X$ is Cartier).

\begin{questionintro}\label{quest:WO-rational_canonical}
    Are $W\cO$--rational (resp.\ $\bF_p$--rational), quasi--Gorenstein normal singularities canonical?
\end{questionintro}

Unfortunately, the answer is negative for both cases of this question, as shown by taking a cone over a supersingular K3 surface. Therefore, some ordinarity condition has to be added, for which we chose $F$--purity:

\begin{questionintro}\label{quest:WO-rational_F-pure_canonical}
    Are $W\cO$--rational (resp.\ $\bF_p$--rational), quasi--Gorenstein normal \textbf{and \textit{F}--pure} singularities canonical?
\end{questionintro}

For $\bF_p$--rationality, we have an immediate positive answer in every dimension:

\begin{propletter}[{\autoref{main_thm_for_F_p-rational}}]\label{prop_intro}
    Let $X$ be a normal, $F$--pure, quasi--Gorenstein variety with $\bF_p$--rational singularities. Then $X$ has canonical singularities.
\end{propletter}

Let us turn now to the question about $W\cO$--rational singularities. It turns out that with the $F$--pure assumption, we only need $\bQ_p$--rational singularities to make our proof work (at least up to dimension $4$). Recall the following diagram relating these different types of singularities:
\begin{align*}
\fbox{$\bF_p$--rational}  \Rightarrow \fbox{$\bQ_p$--rational}
\Leftarrow \fbox{$W\cO$--rational} 
\end{align*}

Let us point out that both converses of the implications above fail (see also \cite[Examples 3.14 and 3.15]{Baudin_Bernasconi_Kawakami_Frobenius_GR_fails}): 
\begin{itemize}
    \item a cone over a singular Enriques surface in characteristic $2$ has $\bQ_p$--rational singularities but not $\bF_p$--rational singularities;
    \item a cone over a supersingular elliptic curve has $\bQ_p$--rational singularities but not $W\cO$--rational singularities.\medskip
\end{itemize}

Our main results are the following affirmative answers up to dimension four:

\begin{thmletter}[{\autoref{main_thm_threefolds}}]\label{Theorem_B}
    Let $X$ be a normal, quasi--Gorenstein, $F$--pure threefold with $\bQ_p$--rational singularities over an $F$--finite field of characteristic $p \neq 2$. Then $X$ has canonical singularities.
\end{thmletter}

\begin{thmletter}[{\autoref{main_thm_foufolds}}]\label{Theorem_C}
    Let $X$ be a normal, quasi--Gorenstein $F$--pure fourfold with $\bQ_p$--rational singularities over a perfect field of characteristic $p > 5$. Assume furthermore that $X$ satisfies \autoref{hyp}. Then $X$ has canonical singularities.
\end{thmletter}

\begin{remintro}\label{remintro}
    \begin{itemize}
        \item \autoref{hyp} is a hypothesis about existence of log resolutions.
        \item The assumption $p \neq 2$ in \autoref{Theorem_B} is optimal, as the result fails for $p = 2$ (see \autoref{example:singular_enriques_surfaces}).
        \item It could be that in the case of $W\cO$--rational singularities, we could weaken the $F$--pure assumption to a quasi--$F$--pure assumption. This would make a lot of sense because $F$--splitting tends to give non--vanishing of $\bQ_p$--cohomology for Calabi--Yau varieties, i.e.\ the slope zero part of $W\cO$-cohomology (see e.g. \autoref{lem:wo_o_cohomology_ordinary}). On the other hand, in the case of $W\cO$--rational singularities, having non--vanishing of $W\cO$--cohomology would be enough. It turns out that quasi--$F$--splitting tends to ensure this (\cite{Yobuko_Quasi_F_splitting_and_lifting_of_CY}), which is why we believe this would be a more natural setup to work with.
        \item In the case of $\bF_p$--rational singularities, quasi--$F$--purity is really not enough, as shown by taking a cone over a supersingular elliptic curve.
     \end{itemize}
\end{remintro}

With the above remark in mind, it would be interesting to have an answer to the following question:

\begin{questionintro}\label{quest:WO-rational_quasi-F-pure_canonical}
    Let $X$ be a normal, quasi--$F$--pure threefold with $W\cO$--rational singularities over a perfect field of characteristic $p \neq 2$. Does $X$ have canonical singularities?
\end{questionintro}

Here is a diagram summarizing the whole picture in positive characteristic:

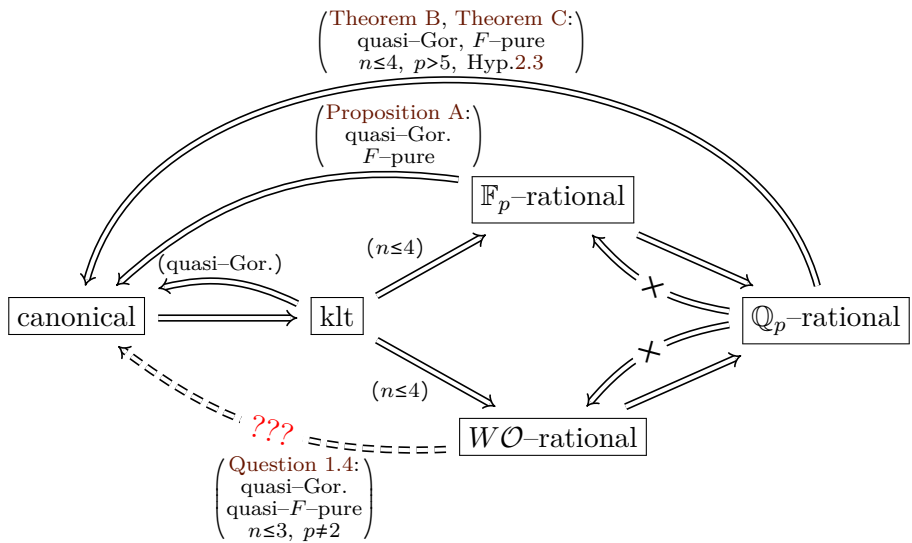
\[\begin{tikzcd}[arrows=Rightarrow]
	&&& {\fbox{$\bF_p$--rational}} \\
	{\fbox{canonical}} && {\fbox{klt}} && {\fbox{$\bQ_p$--rational}} \\
	&&& {\fbox{$W\cO$--rational}}
	\arrow["(n\leq 4)", from=2-3, to=1-4]
	\arrow[from=1-4, to=2-5]
	\arrow["{(n\leq 4)}"', from=2-3, to=3-4]
	\arrow[from=3-4, to=2-5]
	\arrow["{\left(\substack{\text{\autoref{prop_intro}:}\\\text{quasi--Gor.}\\F\text{--pure}}\right)}"'{pos=0.15}, bend right=25, from=1-4, to=2-1]
	\arrow["{\left(\substack{\text{\autoref{quest:WO-rational_quasi-F-pure_canonical}:}\\\text{quasi--Gor.}\\\text{quasi--}F\text{--pure}\\n\leq3,\ p \neq 2}\right)}"{pos=0.17}, "\color{red}{\ ???\ }"' {ddiag,description,font=\large}, bend left=20, dashed, from=3-4, to=2-1]
	\arrow["{\left(\substack{\text{\autoref{Theorem_B}, \autoref{Theorem_C}:}\\\text{quasi--Gor, $F$--pure}\\n\leq 4,\ p>5,\ \text{Hyp.\ref{hyp}}}\right)}"', bend right=75, from=2-5, to=2-1]
	\arrow[from=2-1, to=2-3]
	\arrow["{(\text{quasi--Gor.})}"'{pos=0.57}, bend right=20, from=2-3, to=2-1]
    \arrow[bend left=20, "\times"{anchor=center,sloped,description,font=\large}, from=2-5, to=1-4]
    \arrow[bend right=20, "\times"{anchor=center,sloped,description,font=\large}, from=2-5, to=3-4]
\end{tikzcd}\]

\subsection{Strategy of proofs}

Let us discuss the ideas involved in the proofs of the main theorems. Denote by $\cP_n$ the assertion that every $n$--fold $X$ which is normal, quasi--Gorenstein, $F$--pure, with $\bQ_p$--rational singularities, and defined over an $F$--finite field, has canonical singularities. Showing $\cP_n$ is naturally inductive in $n$: if $\cP_{n-1}$ is established, then by localization, proving $\cP_n$ reduces to excluding isolated, strictly log canonical $n$--fold singularities $x\in X$ which satisfy the given hypotheses.

It is then natural to consider some kind of resolution $\pi\colon Y\to X$ and study the exceptional divisor $E$ over $x$. For surfaces, this will be the \emph{minimal resolution} (to keep $E$ as simple as possible). A higher--dimensional analogue is the \emph{dlt modification} (see \autoref{existence_dlt_modifications}). An important property of these resolutions will be that they satisfy $K_Y+E\sim\pi^{*}K_X$. If we assume for simplicity that $X$ is defined over an algebraically closed field, then $\bQ_p$--rationality is equivalent to $E$ having vanishing higher $\bQ_p$--cohomology. Hence, if we can show that $E$ must have some non--vanishing higher $\bQ_p$--cohomology, then we arrive at the desired contradiction.

A rough idea is the following: by adjunction, $E$ has trivial canonical bundle, so in particular we have $H^{\dim E}(E,\cO_E)\neq 0$. We want to somehow descend this non--vanishing to $\bQ_p$--cohomology, and for this we need $E$ to be arithmetically well--behaved, i.e.\ that the Frobenius acts nicely. This is where the assumption of $F$--purity comes into play: by $F$--adjunction, it will ensure that $E$ is globally $F$--split (or more precisely this is the case of its weak normalization, but let us ignore this issue during the introduction). When $X$ is a surface (i.e.\ $n = 2$), this is already enough to conclude the proof. 

Let us now discuss what happens for $n=3,4$. Interestingly, the threefold and fourfold cases are of very different flavour.\medskip

\emph{The threefold case:} For $X$ a threefold, $E=\bigcup_i E_i$ is a union of integral surfaces. The $\bQ_p$--cohomology of $E$ is then made up of two things: the $\bQ_p$--cohomology of the individual components $E_i$, and the singular cohomology of the dual complex $\cD(E)$. The latter is a topological space encoding how the $E_i$'s are glued together to form $E$. By carefully describing the individual components using adjunction and classical surface arguments (\autoref{lem:dlt_blowup_gorenstein_comps_of_exc} and \autoref{lem:classification_dlt_surfaces}), and then deducing restrictions on the global structure of the dual complex (\autoref{prop:dlt_blowup_gorenstein_dual_complex}), we obtain a list of four cases which need to be checked: $\cD(E)$ is a point, a cycle, a path graph, or a $2$--dimensional compact connected topological manifold. Looking at $\bQ_p$--cohomology then automatically eliminates the first three cases (provided that $p \neq 2$), and yields that the compact topological surface in the last case must be the real projective plane $\bR\bP^2$. However, it turns out that $\bQ_p$--cohomology alone is not enough to finish the proof. Luckily, through the global $F$--splitness of $E$ we also have a lot of information about its $\bF_p$--cohomology, and we know in particular that $H_{\et}^2(E,\bF_p)\neq 0$. This allows us to conclude that the real projective plane case cannot happen.\medskip

\emph{The fourfold case:} Note that $E$ is an odd--dimensional scheme with trivial canonical bundle. We can then show that $E$ is Cohen--Macaulay using results of \cite{Bernasconi_Kollar_Vanishing_theorems_for_threefolds_in_char_p}, implying that $\chi(E,\cO_E)=0$ by Serre duality. On the other hand, by $\bQ_p$--rationality we have in particular that $\chi_{\et}(E,\bQ_p)=1$. The key trick is to use then that $\chi_{\et}(E,\bQ_p)=\chi_{\et}(E,\bF_p)$ (see \autoref{lem:euler_char}). By the global $F$--splitness of $E$, the latter is actually equal to the usual Euler characteristic $\chi(E,\cO_E)$, so we obtain a contradiction.  

\subsection{Notation and basic definitions}

\begin{itemize}
    \item For $n \geq 1$ an integer, we write $[n] \coloneqq \{1, \dots, n\}$.

    \item An $\bF_p$--scheme $X$ is $F$--finite if its absolute Frobenius morphism is finite.

    \item A variety is a scheme $X$, separated and essentially of finite type over an $F$--finite field of characteristic $p>0$.

    \item We say that a variety is a curve (resp.\ surface, resp.\ $n$--fold) if it has dimension one (resp.\ two, resp.\ $n$).

    \item If $X$ is a variety over a field $k$ with structural morphism $f \colon X \to \Spec k$, we define $\omega_X^{\bullet} \coloneqq f^!\cO_{\Spec k}$ (\cite[5.2]{Nayak_Compactification_for_essentially_finite_type_maps}). We also define the canonical sheaf $\omega_X$ as the first non--zero cohomology sheaf of $\omega_X^{\bullet}$. If $X$ is normal, we fix a Weil divisor $K_X$, such that $\cO_X(K_X) \cong \omega_X$. Such a divisor is called a \emph{canonical divisor}.
    
    \item A variety is \emph{quasi--Gorenstein} if its canonical module $\omega_X$ is a line bundle.

    \item Given a scheme $X$, its reduction is denoted $X_{\red}$. If $X$ is integral, then $X^{\nu}$ denotes its normalization.

    \item If $X$ is an excellent Noetherian scheme, we denote its weak normalization (see \cite[Definition I.7.2.1]{Kollar_Rational_curves_on_algebraic_varieties}) by $\wn \colon X^{\wn} \to X$.

    \item A universal homeomorphism is a morphism of schemes which is a homeomorphism after arbitrary base change. Important examples are the reduction map, a purely inseparable morphism of normal schemes or the weak normalization (see \autoref{lem:weak_nor_univ_homeo}).

    \item Let $X$ be a normal $F$--finite $\bF_p$--scheme. Then $X$ is globally $F$--split (resp.\ $F$--pure) if the $p$--power map $F \colon \cO_X \to F_*\cO_X$ splits (resp.\ locally splits).

    More generally, given an effective Weil divisor $\Delta$ on $X$, we say that the pair $(X, \Delta)$ is globally sharply $F$--split (resp.\ sharply $F$--pure) if there exists $e > 0$ such that the composition  morphism  
    \[\begin{tikzcd}
        \cO_X \arrow[rr, "F^e"] &  & F^e_*\cO_X \arrow[rr, "\inc"] &  & F^e_*\cO_X((p^e - 1)\Delta)
    \end{tikzcd}\] splits (resp.\ locally splits).

    \item On an $\bF_p$--scheme $X$, we define the perfection of $\mathcal{O}_X$ as 
    \[ \cO_X^{1/p^{\infty}} \coloneqq \varinjlim \: (\cO_X \xrightarrow{F} F_*\cO_X \xrightarrow{F} F^2_*\cO_X \xrightarrow{F} \dots) \] (see \cite{Bhatt_Lurie_A_Riemann-Hilbert_correspondence_in_positive_characteristic}). 

    \item Given a scheme $X$ in characteristic $p > 0$ and $n \geq 1$, the sheaf $W_n\cO_X$ denotes the usual $n$-truncated Witt vectors over $\cO_X$ (see e.g. \cite[Section 2.6]{Serre_Local_fields}).
    
\end{itemize}


\subsection{Acknowledgements}

We would like to thank Fabio Bernasconi, Stefano Filipazzi, Léo Navarro Chafloque and Jakub Witaszek for useful discussions or comments on the content of this article. We would also like to thank the anonymous referee for their careful reading and thoughtful comments to improve the exposition of the article.

JB and LR were partly supported by the  project grant \#200020B/192035 from the Swiss National Science
Foundation (FNS), as well as by the ERC starting grant \#804334. MZ was partially supported by the FNS project grant \#200021/169639.
ZsP was partially supported by FNS project grants \#200020B/192035 and \#200021/169639, as well as by the ERC starting grant \#804334.



\section{Preliminaries}

\subsection{Minimal model program, dlt singularities and weak normalizations}\label{subsection:MMP}

Let $X$ be a normal variety, and let $\Delta$ be an effective $\bQ$--divisor with coefficients $\leq 1$ such that $K_X + \Delta$ is $\bQ$--Cartier. We recall the definition of singularities for the pair $(X, \Delta)$ according to the MMP.
Given a proper birational morphism $\pi \colon Y \to X$ with $Y$ normal, we choose a canonical divisor $K_Y$ such that $\pi_*K_Y=K_X$, and we write
$$ K_Y + \pi_*^{-1}\Delta = \pi^*(K_X+\Delta)+\sum_i a(E_i, X, \Delta) E_i, $$
where $E_i$ run through the prime exceptional divisors and $\pi_*^{-1}\Delta $ is the strict transform of $\Delta$. 
The coefficient $a(E, X, \Delta)$ is called the \emph{discrepancy} of $(X, \Delta)$ along $E$, and it depends exclusively on the divisorial valutation associated to $E$.

\begin{definition}
    We say that $(X, \Delta)$ is \emph{canonical} (resp.\ klt, resp.\ log canonical) if for every $E$ as above, $a(E, X, \Delta) \geq 0$ (resp.\ $a(E, X, \Delta) > -1$ and $\lfloor \Delta \rfloor = 0$, resp.\ $a(E, X, \Delta) \geq -1$). 
    
    We say that $(X, \Delta)$ is \emph{strictly log canonical} if it is log canonical and not klt.
\end{definition}

Let us now go to dlt singularities. Given a log canonical pair $(X, \Delta)$, we say that a divisor $E$ on a birational model of $X$ is a \emph{log canonical place} if it is a divisorial valuation of discrepancy $-1$. The image of $E$ on $X$ is called a \emph{log canonical center}, and denoted $\cent_X(E)$.

\begin{definition}
    We say that $(X, \Delta)$ is \emph{dlt} if $(X, \Delta)$ is log canonical, and for every log canonical place $E$, the generic point of $\cent_X(E)$ is contained in the snc locus of $(X, \Delta)$.
\end{definition} 

Dlt singularities will show up because we want to take \emph{dlt modifications} (see \autoref{existence_dlt_modifications}), which is roughly a higher--dimensional analogue of minimal resolutions of surfaces. When dealing with fourfolds, we will need the following hypothesis on the existence of resolutions of singularities.

\begin{hypothesis}
\label{hyp}
Following \cite{Hacon_Witaszek_On_the_relative_MMP_for_4folds_in_pos_and_mixed_char}, we assume that, given an integral fourfold $X$ over a perfect field $k$, log resolutions of all log pairs with the underlying variety being birational to $X$ exist and are given by a sequence of blowups along the non--snc locus.
\end{hypothesis}

\begin{thm}\label{existence_dlt_modifications}
    Let $n \leq 4$, and let $X$ be a log canonical $n$--fold. If $n = 4$, assume that $X$ is of finite type over a perfect field $k$ of characteristic $p > 5$ and satisfies \autoref{hyp}.

    Then there exists a dlt pair $(Y, E = \sum_i E_i)$, and a birational morphism $\pi \colon Y \to X$, such that 
    \[ K_Y + E \sim_{\bQ} \pi^*K_X ,  \] where all $E_i$'s are $\bQ$--Cartier prime divisors, and $E = \bigcup_i E_i$ is the exceptional locus of $\pi$.
\end{thm}

\begin{remark}
    \begin{itemize}
        \item The pair $(Y, E)$ above is called a \emph{dlt modification} of $X$.
        \item As will be clear from the proof, if $X$ was already quasi--Gorenstein to begin with (i.e.\ $K_X$ is Cartier), then we actually have $K_Y + E \sim \pi^*K_X$.
    \end{itemize}
\end{remark}

\begin{proof}
    Note that since we did not assume that $X$ was $\bQ$--factorial, we will need to use the non--$\bQ$--factorial MMP of Koll\'ar's, where we do not only contract extremal rays, but also extremal faces (see \cite{Kollar_MMP_without_Q_factoriality}).

    We will only prove the result for $n = 4$. The cases $n \leq 2$ are standard, and for $n = 3$, replace in the proof below the use of \cite[Theorem 2.20]{Baudin_Bernasconi_Kawakami_Frobenius_GR_fails} by \cite[Theorem 9]{Kollar_MMP_without_Q_factoriality}.

    Let $f \colon Z \to X$ be a log resolution, with exceptional divisor $E' = \Exc(f) = \bigcup_i E'_i$. By \cite{Kollar_Witaszek_resolution_with_ample_divisor}, we may assume that $E'$ supports an $f$--ample divisor. Hence, we can find a divisor $H$ satisfying the assumptions (iii), (iv) and (v) of \cite[Theorem 2.20]{Baudin_Bernasconi_Kawakami_Frobenius_GR_fails} with the pair $(Z, E')$. Since (i) and (ii) are satisfied by our hypotheses, we can apply \emph{loc.\ cit.} to run a non--$\bQ$--factorial $(K_Z + E')$--MMP over $X$. We then end up with a dlt pair $(Y, E)$ and a birational projective morphism $\pi \colon Y \to X$. Let us list the important properties of our dlt modification, all of which follow from \emph{loc.\ cit.}:
    
    \begin{enumerate}
        \item\label{exc_supports_ample} $\Exc(\pi)$ supports an anti--effective exceptional $\pi$--ample divisor $H$;
        \item \label{div_part_of_exc_is_E} the divisorial part of $\Exc(\pi)$ is exactly $E$ (this follows by definition of the MMP--steps);
        \item each irreducible component of $E$ is $\bQ$--Cartier.
    \end{enumerate}

    It follows from \autoref{exc_supports_ample} and \autoref{div_part_of_exc_is_E} that in fact, $E = \Exc(\pi)$. Indeed, if $C$ is a contracted irreducible curve, then $H\cdot C > 0$, so by \autoref{exc_supports_ample} there is some irreducible component $E_i$ of $E$ such that $E_i \cdot C < 0$, whence $C \subseteq E_i$.

    Since $X$ is log canonical and $E = \Exc(\pi)$, we know that \[ K_Y + E \sim_{\bQ} \pi^*K_X + F, \] with $F$ an effective exceptional $\bQ$--divisor. Since it must be $\pi$--nef by definition of a relative minimal model, we conclude that $F = 0$ by the negativity lemma. 
\end{proof}

\begin{remark}
    We cited \cite{Baudin_Bernasconi_Kawakami_Frobenius_GR_fails} (see also \cite{Kollar_MMP_without_Q_factoriality}) because we want to work with varieties which are not necessarily $\bQ$--factorial. If we stick to $\bQ$--factorial ones, then the right references become \cite{Hacon_Witaszek_On_the_relative_MMP_for_threefolds_in_low_char} for threefolds, and \cite{Hacon_Witaszek_On_the_relative_MMP_for_4folds_in_pos_and_mixed_char} for fourfolds.
\end{remark}

Let us now study dlt pairs in greater detail.  In the proofs of the main theorems, we will be interested in the $\bQ_p$-- and $\bF_p$--cohomology of $E$ where $(Y,E)$ is a dlt modification of $X$. Note that log canonical centers of dlt pairs may not be normal (\cite{Cascini_Tanaka_Purely_log_terminal_3folds_with_non_normal_centers_in_char_2, Bernasconi_Non-normal_purely_log_terminal_centres}), which at first sight might seem problematic. However, since $\bQ_p$-- and $\bF_p$--cohomology is topological in nature by \autoref{lem:topological_invariance_of_etale_cohomology}, we may replace $E$ by any scheme universally homeomorphic to $E$. A useful gadget for this purpose is the weak normalization $\wn\colon E^{\wn}\to E$  as it is a universal homeomorphism. We quickly restate the proof of this well--known fact, and then collect useful properties of $E^{\wn}$ (see \autoref{cor:strata_of_dlt_locus_is_S_2_up_to_univ_homeo}).

\begin{lemma}\label{lem:weak_nor_univ_homeo}
    Let $X$ be an excellent Noetherian scheme. Then its weak normalization $\wn \colon X^{\wn} \to X$ is a universal homeomorphism.
\end{lemma}
\begin{proof}
    Combining properties $(7.2.1.1.1)$ and $(7.2.1.1.2^{\mathrm{weak}})$ in \cite[Definition I.7.2.1]{Kollar_Rational_curves_on_algebraic_varieties} and \stacksproj{01S3}, the weak normalization map is finite, radicial and bijective. We then conclude by \stacksproj{01S4} and \stacksproj{04DF}.
\end{proof}

We state the next lemma for local complete intersection as well, although only the case of a $\bQ$--Cartier divisor will be needed. Due to its generality, it is rather surprising to the authors.

\begin{lemma}[\protect{\cite[Lemma 2.1]{Hacon_Witaszek_On_the_relative_MMP_for_4folds_in_pos_and_mixed_char}}]\label{lem:hypersurface_of_normal_var_is_S_2_up_to_nilp}
    Let $X$ be an excellent Noetherian normal scheme and let $D\subseteq X$ be a $\bQ$--Cartier divisor. Then the weak normalization $D^{\wn}$ of $D$ is $S_2$. The same conclusion holds if $D$ is a local complete intersection.
\end{lemma}

\begin{proof}
    To start, if $D$ is $\bQ$--Cartier, then as $D$ has the same reduction as any Cartier multiple and the weak normalization factors through the reduction, we may assume that $D$ itself is Cartier. Hence it suffices to treat the case where $D$ is an lci.

    Throughout, we will use the terminology and results from \cite{Kollar_variants_of_normality_for_Noetherian_schemes}. By \cite[Corollary 8]{Kollar_variants_of_normality_for_Noetherian_schemes}, it suffices to show that every pair $Z\subseteq D^{\wn}$ with $\codim_{D^{\wn}}(Z)\geq 2$ is normal. Given that these pairs are weakly normal by \cite[Proposition 5]{Kollar_variants_of_normality_for_Noetherian_schemes}, it is enough to show that they are topologically normal. Furthermore, as $\wn\colon D^{\wn}\to D$ is a universal homeomorphism by \autoref{lem:weak_nor_univ_homeo}, it is enough to show that every pair $Z \inc D$ with $\codim_D(Z) \geq 2$ is topologically normal (we are using that strong topological normality and topological normality are equivalent, see \cite[Theorem 31]{Kollar_variants_of_normality_for_Noetherian_schemes}).
    
    By \cite[Theorem 41]{Kollar_variants_of_normality_for_Noetherian_schemes}, we need to prove that for any $x\in D$ with $\codim_D x\geq 2$, the scheme $D^{\sh}\setminus\{x^{\sh}\}$ is connected, where $(D^{\sh}, x^{\sh})$ denotes the strict henselization of $D$ at $x$. The strategy will be to use a version of Faltings' connectedness theorem (namely \stacksproj{0ECR}, see also \cite[Theorem XIII.2.1]{Grothendieck_SGA_2}), which however only applies in the complete case. Hence, we further pass to the completion $(\tD, \tx)$ of $D^{\sh}$ at $x^{\sh}$. It then suffices to prove that $\tD\setminus\{\tx\}$ is connected, since the map $\tD\to D^{\sh}$ is faithfully flat (hence surjective) and the unique preimage of $x^{\sh}$ is $\tx$.
     
    Denote by $X^{\sh}$ the strict henselization of $X$ at $x$ and by $\tX$ the completion of $X^{\sh}$ at $x^{\sh}$. Also, denote by $c$ the codimension at $x$ of $D$ in $X$. We now verify that $c$, $\tx$, $\tD$ and $\tX$ satisfy the hypotheses of the aformentioned connectedness theorem.

    First of all, let us show that $\tX$ is Noetherian normal of dimension $\dim\tX\geq c+2$. By \stacksproj{06LJ} and \stacksproj{06DI}, $X^{\sh}$ is Noetherian normal. By point (3) of \stacksproj{04GP} and point (2) of \cite[Theorem 1.3]{Bhatt_Scholze_The_pro-etale_topology_for_schemes}, the natural map $X^{\sh}\to \Spec\cO_{X,x}$ is weakly \'etale, so $X^{\sh}$ is also excellent (see \cite[Theorem 8.1]{Raynaud_Anneaux_excellents}). By \stacksproj{0C23}, we obtain that $\tX$ is normal, and by \stacksproj{06LJ} it is also Noetherian. By \stacksproj{07NV} and \stacksproj{06LK}, we have
    \begin{align*}
        \dim\tX = \dim X^{\sh} = \dim \cO_{X,x}\expl{=}{$X$ is catenary by excellence}c+\codim_D x\geq c+2.
    \end{align*}

    Second of all, we describe the ideal of $\tD$ in $\tX$. As $D$ is an lci of codimension $c$ at $x$, the ideal of $D$ in $\cO_{X,x}$ can be generated by $c$ elements $f_1,\ldots,f_c\in \fm_{X,x}$. Applying \stacksproj{08HV} to the quotient map $\cO_{X,x}\to\cO_{X,x}/(f_1,\ldots,f_c)$ and the respective maximal ideals, we obtain a natural isomorphism
        \begin{align*}
            (\cO_{X,x}/(f_1,\ldots,f_c))^{\sh}\cong (\cO_{X,x}^{\sh}/(f_1,\ldots,f_c))^{\sh},
        \end{align*}
    where on the right we take the strict henselization with respect to the maximal ideal (the image of $f_i$ in $\cO_{X,x}^{\sh}$ is still denoted $f_i$). But as henselianity is preserved under quotients (if solutions lift to the ring then they certainly lift to a quotient) and as $\cO_{X,x}^{\sh}/(f_1,\ldots,f_c)$ and $\cO_{X,x}^{\sh}$ have the same residue field, $\cO_{X,x}^{\sh}/(f_1,\ldots,f_c)$ is already strictly henselian. Hence $(f_1,\ldots,f_c)$ is the ideal of $D^{\sh}$ in $X^{\sh}$. Finally, as $\cO_{X,x}^{\sh}$ is Noetherian local, completion with respect to $\fm^{\sh}_{X,x}$ is exact. Hence we obtain that the ideal $I_{\tD}$ of $\tD$ in $\tX$ is also given by $(f_1,\ldots,f_c)$.
    
    We are now in place to prove that $\tD\setminus\{\tx\}$ is connected. By \stacksproj{0ECR}, we must show that the cohomological dimension of $I_{\tD}$ in $\cO_{\tX,\tx}$ is at most $\dim \tX-2$ (as $\tX$ is normal local and hence irreducible). But $\dim \tX\geq c+2$, and by \stacksproj{0DX9}, the cohomological dimension of $I_{\tD}$ in $\cO_{\tX,\tx}$ is smaller than or equal to the number of generators of $I_{\tD}$, which is $c$. Hence the hypotheses of \stacksproj{0ECR} are verified, and we conclude that $\tD\setminus\{\tx\}$ is connected.
\end{proof}

\begin{corollary}[{\cite[Lemma 2.1]{Hacon_Witaszek_On_the_relative_MMP_for_4folds_in_pos_and_mixed_char}}]\label{cor:strata_of_dlt_locus_is_S_2_up_to_univ_homeo}
    Let $(X, D + \Delta)$ be a dlt pair with $D$ prime and $\bQ$--Cartier. Then $D^{\wn}$ is normal, i.e. the normalization $D^{\nu} \to D$ is a universal homeomorphism.
\end{corollary}

\begin{proof}
    By \autoref{lem:hypersurface_of_normal_var_is_S_2_up_to_nilp}, it is enough to show that $D$ is weakly normal at codimension one points. This is immediate by the dlt assumption.
\end{proof}

We now have all the ingredients to prove the following general fact about arbitrary dimensional dlt pairs. The characteristic zero version is well--known, see \cite[Theorem 4.16]{Kollar_Singularities_of_the_minimal_model_program}.

\begin{proposition}\label{prop:dlt_structure}
Let $(Y,E)$ be a dlt pair, Write $E = \sum_{i = 1}^r E_i$ with each $E_i$ a prime divisor. Assume furthermore that each $E_i$ is $\bQ$--Cartier. The following assertions are true:
\begin{enumerate}
        \item \label{prop:dlt_struct_geometry}
            Each intersection $E_{i_1} \cap \dots \cap E_{i_s}$  with $1 \leq i_1 < \dots < i_s \leq r$ has pure codimension $s$, and is normal up to universal homeomorphism. In particular, it is the disjoint union of its connected components.
	\item \label{prop:dlt_structure_it1}
		the $s$--codimensional log canonical centers of $(Y,E)$ are exactly the irreducible components of intersections 
	\[
	E_{i_1} \cap \ldots \cap E_{i_s},
	\] for some choice of indices $1 \leq i_1<\cdots<i_s \leq n$.
    \item\label{prop:dlt_structure_adjunction} For all $i$, we have $\Diff(E-E_i)=(E-E_i)|_{E_i^{\nu}}$ and the pair $(E_i^{\nu},(E-E_i)|_{E_i^{\nu}})$ is dlt.
	\item \label{prop:dlt_structure_it3}
		$Y$ is regular along the generic points of $E_i \cap E_j$, for $i \neq j$.
\end{enumerate}
If $K_Y+E$ is Cartier then additionally:
\begin{enumerate}[resume]
\item \label{prop:dlt_structure_it4} 
	$Y$ is regular in codimension two along $E$,
\item \label{prop:dlt_structure_it5}
	the pairs $(E_i^{\nu},(E-E_i)|_{E_i^{\nu}})$ are dlt, and every $E_i^{\nu}$ is regular at its codimension two points in $(E - E_i)|_{E_i^{\nu}}$. In particular, each $E_i^{\nu}$ has canonical singularities at its codimension two points.
\end{enumerate}
\end{proposition}

\begin{proof}
    Let us first show the statements which do not require $K_Y + E$ to be Cartier. The idea is to mimic the inductional argument in \cite[Theorem 4.16]{Kollar_Singularities_of_the_minimal_model_program}.

    Let us first show that every irreducible component of each $E_i|_{E_j^{\nu}}$ is $\bQ$--Cartier. This would follow from our $\bQ$--Cartier assumption if we could show that $E_i \cap E_j$ was the disjoint union of its irreducible components. First of all, since $E_i$ and $E_j$ are $\bQ$--Cartier, their intersection $E_i \cap E_j$ must be of pure codimension two. Given that $(X, E_i + E_j)$ is dlt, we know by adjunction (see \cite[Claim 4.16.4]{Kollar_Singularities_of_the_minimal_model_program}) that $(E_i^{\nu}, E_j|_{E_i^{\nu}})$ is also dlt. We then deduce by \autoref{cor:strata_of_dlt_locus_is_S_2_up_to_univ_homeo} that the weak normalization of $E_j|_{E_i^{\nu}}$ is $S_2$, so its irreducible components must either not meet, or meet at codimension one points of $E_j|_{E_i^{\nu}}$. Since $E_i^{\nu} \to E_i$ is a universal homeomorphism (again by \autoref{cor:strata_of_dlt_locus_is_S_2_up_to_univ_homeo}), we deduce that two irreducible components of $E_i \cap E_j$ must meet at a codimension two point of $Y$. Localizing at this point and hence reducing to a dlt surface pair, we see that this is impossible (this intersection would give rise to a non--snc log canonical center).

    We are now in place to apply exactly the same argument as in the proof of \cite[Theorem 4.16]{Kollar_Singularities_of_the_minimal_model_program}, replacing each instance of ``normal'' with ``normal up to universal homeomorphism'' by \autoref{cor:strata_of_dlt_locus_is_S_2_up_to_univ_homeo}. This immediately concludes the proof of parts \autoref{prop:dlt_struct_geometry}, \autoref{prop:dlt_structure_it1} and \autoref{prop:dlt_structure_adjunction}. Furthermore, since $E_i \cap E_j$ has pure codimension two, the pair $(Y, E)$ is snc along its generic points. Hence, part \autoref{prop:dlt_structure_it3} follows.

    Let us now assume that $K_Y + E$ is Cartier. In order to prove \autoref{prop:dlt_structure_it4}, we localize $(Y, E)$ at the codimension two point. This way we obtain a dlt surface pair $(T, D)$ with $K_T + D$ Cartier, and a distinguished point $t \in D$. We need to show that $T$ is regular at $t$, which we will show by contradiction. We first claim that the discrepancies of the minimal resolution over the point $t$ are negative. Indeed, they are at most zero because $T$ is singular at $t$ and consequently negative since $D$ passes through $t$. Since $K_T + D$ is Cartier, the discrepancies are then equal to $-1$. However, by definition of dlt, this means that $t$ is an snc point of $(T,D)$ which gives a contradiction. 

    The argument that $\Diff(E-E_i)=(E - E_i)|_{E_i^{\nu}}$ and that the pair $\left(E_i^{\nu},(E - E_i)|_{E_i^{\nu}}\right)$ is dlt is exactly as in the proof of \cite[Claim 4.16.4]{Kollar_Singularities_of_the_minimal_model_program}. Note also that by adjunction, $K_{E_i^{\nu}} + (E - E_i)|_{E^{\nu}}$ is Cartier. The claim that $E_i^{\nu}$ is regular at every codimension two point in $(E - E_i)|_{E_i^{\nu}}$ follows from the same argument as in the previous part. It is then immediate that $E_i^{\nu}$ has canonical singularities in codimension two, since outside of $(E - E_i)|_{E_i^{\nu}}$, the variety $E_i^{\nu}$ is dlt and $K_{E_i^{\nu}}$ is Cartier.
\end{proof}

The situation in the following corollary will exactly be what will happen in our main theorems in dimension $3$ and $4$.

\begin{corollary}\label{cor:full_exc_div_GFS_and_K_trivial}
    Let $X$ be a normal, quasi--Gorenstein, $F$--pure variety. Let $\pi \colon Y \to X$ be proper birational morphism. Assume that $\pi$ is an isomorphism away from a finite number of points on $X$, and that we can write \[ K_Y + E \sim \pi^*K_X, \] where $(Y, E)$ is dlt and every component of $E$ is $\bQ$--Cartier. Then the following holds:
    \begin{enumerate}
        \item\label{F_adj_total} $E^{\wn}$ is globally $F$--split, and $\omega_{E^{\wn}} \cong \cO_{E^{\wn}}$;
        \item\label{F_adj_cpnts} for any component $E_i$ of $E$, the pair $(E_i^{\nu}, \Delta_i \coloneqq (E - E_i)|_{E_i^{\nu}})$ is dlt, globally sharply $F$--split, and satisfies $K_{E_i^{\nu}} + \Delta_i \sim 0$.
    \end{enumerate}
\end{corollary}
\begin{proof}
    Since $K_X$ is Cartier, so is $K_Y + E$. We then deduce by \autoref{prop:dlt_structure}.\autoref{prop:dlt_structure_it4} that $Y$ is regular in codimension two along $E$. In particular, $E$ is Cartier along its codimension one points. Furthermore, since $X$ is $F$--pure, we can lift the $F$--splitting section to $Y$ and deduce that $(Y, E)$ is globally sharply $F$--split (see \cite[Lemma 3.3]{Gongyo_Takagi_Surfaces_of_gloally_F_regular_and_F_split_type}).
    
    Let $U \coloneqq Y_{\reg} \cap E$, and let us show that $U$ is weakly normal, by showing that for $y \in U$ such that $\overline{y}$ is nowhere dense, the pair $(y, U_y)$ is weakly normal (by \cite[Proposition 10]{Kollar_variants_of_normality_for_Noetherian_schemes}, this will conclude). If $y$ has codimension $1$, then $(Y, E)$ is snc at $y$ by the dlt assumption, so $(y, U_y)$ is weakly normal. If $y$ has codimension at least $2$, then $(y, U_y)$ is also weakly normal by \cite[Corollary 8]{Kollar_variants_of_normality_for_Noetherian_schemes} since $U_y$ is $S_2$ (it is a divisor in a regular scheme). Thus, $U$ is indeed weakly normal. In other words, the following holds:
    \begin{enumerate}[(a)]
        \item $U$ contains all codimension one points of $E$ (see \autoref{prop:dlt_structure}.\autoref{prop:dlt_structure_it4}); 
        \item $E^{\wn} \to E$ is an isomorphism over $U$ (by our work above and \cite[Proposition 11]{Kollar_variants_of_normality_for_Noetherian_schemes});
        \item $\omega_U \cong \cO_U$ by adjunction and the fact that $\pi$ is an isomorphism away from a finite number of points on $X$; 
        \item $U$ is globally $F$--split by $F$--adjunction (see \autoref{rem:F_adj}).
    \end{enumerate}

    Since $E^{\wn}$ is $S_2$ by \autoref{lem:hypersurface_of_normal_var_is_S_2_up_to_nilp}, we automatically deduce the result for $E^{\wn}$ (i.e.\ \autoref{F_adj_total} holds). The proof of \autoref{F_adj_cpnts} is identical (use \autoref{prop:dlt_structure}.\autoref{prop:dlt_structure_adjunction} to obtain the dlt statement).
\end{proof}

\begin{remark}\label{rem:F_adj}
    By $F$--adjunction, we meant the following standard statement: let $V$ be a smooth variety, and let $D$ and $\Delta$ be Weil divisors such that $(V, D + \Delta)$ is globally sharply $F$--split. Assume either that $D$ is also smooth, or that $\Delta = 0$. Then $(D, \Delta|_D)$ is also globally sharply $F$--split. 
\end{remark}

\subsection{Positive characteristic cohomologies and variants of rational singularities}

Let $f \colon Y \to X$ be a morphism of $\bF_p$--schemes. Throughout, whenever we mention the sheaf $\bF_{p, Y}$ on $Y$, we mean it as an étale sheaf (i.e.\ on the étale site of $Y$). The étale sheaves $R^if_*\bF_{p, Y}$ are then defined by taking the usual higher direct image functor (in the étale site). We will often omit the index $Y$ in $\bF_{p, Y}$, when the context is clear.

For the definition of the sheaf $\bQ_p$ and its higher pushforwards, we refer the reader to \cite[Section 3.3]{Patakfalvi_Zdanowicz_Ordinary}.

\begin{warning}\label{warning:etale_cohomology}
    If $f \colon X \to \Spec k$ is a morphism and $k$ is not separably closed, there is no reason to have an isomorphism \[ R^if_*\bF_p \cong H^i_{\et}(X, \bF_p) \] (the problem being that the étale site of $\Spec k$ is non--trivial). However, this is the case if $k$ is separably closed, as then for any étale sheaf $\cG$ on $\Spec k$ and $i > 0$, $H^i_{\et}(\Spec k, \cG) = 0$ so the Leray spectral sequence \[ H^i_{\et}(\Spec k, R^jf_*\bF_p) \implies H^{i + j}_{\et}(X, \bF_p) \] degenerates. The same also works for $\bQ_p$.
\end{warning}

First of all, let us recall a very important fact of $\bF_p$-- and $\bQ_p$--cohomology.

\begin{lemma}\label{lem:topological_invariance_of_etale_cohomology}
    Let $f \colon Y \to X$ be a universal homeomorphism. Then for all $i \geq 0$, the induced maps $H^i_{\et}(X, \bF_p) \to H^i_{\et}(Y, \bF_p)$ and $H^i_{\et}(X, \bQ_p) \to H^i_{\et}(Y, \bQ_p)$ are isomorphisms. 
    
    If $f$ is finite, then also the natural map $\cO_X^{1/p^{\infty}} \to f_*\cO_Y^{1/p^{\infty}}$ is an isomorphism.
\end{lemma}

\begin{proof}
    The statement about cohomology groups is immediate from \stacksproj{03SI}. To obtain the last statement, note that it is immediate that $\bF_{p, X} \to f_*\bF_{p, Y}$ is an isomorphism. We then conclude by the Riemann--Hilbert correspondence (\cite[Theorem 10.2.7 and Corollary 10.5.6]{Bhatt_Lurie_A_Riemann-Hilbert_correspondence_in_positive_characteristic}).
\end{proof}

\begin{lemma}\label{lem:wo_o_cohomology_ordinary}
    Let $X$ be a proper variety of dimension $d$ defined over an algebraically closed field $k$ of characteristic $p>0$. Assume that the natural Frobenius action on $H^d(X,\cO_X)$ is bijective, and that $H^{d-1}(X,W_n\cO_X) \to H^{d-1}(X,W_{n-1}\cO_X)$ is surjective for every $n$. Then \[ \dim_{\bQ_p}H^d_{\et}(X, \bQ_p) = \dim_k H^d(X, \cO_X). \]
\end{lemma}

\begin{proof}
    Although \cite[Proposition 4.1]{Patakfalvi_Zdanowicz_Ordinary} claims to only prove the case when $H^{d - 1}(X, \cO_X) = 0$, their proof also works under our weaker hypothesis.
\end{proof}

\begin{remark}\label{rem:auotmatic_hyp_wo_o_cohomology_ordinary}
    The surjectivity assumption in Witt cohomology groups in \autoref{lem:wo_o_cohomology_ordinary} is automatic when $X$ is a curve, or when $H^{d - 1}(X, \cO_X) = 0$.
\end{remark}

\begin{lemma}\label{lem:proper_base_change}
    Let \[\begin{tikzcd}
    Y \arrow[r, "g'"] \arrow[d, "f'"'] & X \arrow[d, "f"]                \\
    T \arrow[r, "g"']                   & S                
    \end{tikzcd} \] be a pullback diagram of schemes, where $f$ is proper. Then for all $i \geq 0$, there are natural isomorphisms \[ g^*R^if_*\bF_p \cong R^if'_*\bF_p. \] Furthermore, the same result holds for $\bF_p$ replaced by $\bQ_p$. 

    In particular, if $\pi \colon Y \to X$ is a birational proper morphism with exceptional locus $E \subseteq Y$, then for all $i > 0$, \[ R^i\pi_*\bF_p = R^i\pi_{E, *} \bF_p \] (and similarly for $\bQ_p$), where $\pi_E \colon E \to \pi(E)$ is the induced map.
\end{lemma}

\begin{proof}
    The last statement is an immediate consequence of the first statement about the base change isomorphism. Furthermore, by definition of $\bQ_p$, it is enough to show the base change isomorphism for the sheaves $\bZ/p^n\bZ$. This is a particular case of \stacksproj{095T}.
\end{proof}

Let us recall the following well--known fact.

\begin{lemma}\label{lem:baby_Riemann_Hilbert}
    Let $k$ be a separably closed field of characteristic $p > 0$, and let $X$ be a proper $k$--scheme. Then for all $i \geq 0$ there is a natural inclusion $H^i_{\et}(X, \bF_p) \inc H^i(X, \cO_X)$. If $k = H^0(X, \cO_X)$ and $X$ is globally $F$--split, then this inclusion induces an isomorphism \[ H^i_{\et}(X, \bF_p) \otimes_{\bF_p} k \cong H^i(X, \cO_X). \] In particular, $\dim_k H^i(X, \cO_X) = \dim_{\bF_p}H^i_{\et}(X, \bF_p)$.
\end{lemma}
\begin{proof}
    By \cite[Theorem 10.5.5]{Bhatt_Lurie_A_Riemann-Hilbert_correspondence_in_positive_characteristic} and the fact that $k$ is separably closed (so its étale site is trivial), we conclude that \[ H^i_{\et}(X, \bF_p) \cong \set{x \in H^i(X, \cO_X^{1/p^{\infty}})}{\phi(x) = x}, \] where $\phi$ denotes the action induced by the Frobenius on $\cO_X$. Note that \[ \set{x \in H^i(X, \cO_X^{1/p^{\infty}})}{\phi(x) = x} = \set{x \in H^i(X, \cO_X)}{\phi(x) = x}, \] by construction of $\cO_X^{1/p^{\infty}}$ (see also \cite[Proposition 3.2.9]{Bhatt_Lurie_A_Riemann-Hilbert_correspondence_in_positive_characteristic}), so the inclusion part of the statement is proven.
    
    To prove the second part of the statement, note that we may assume that $k$ is algebraically closed. Indeed, $\bF_p$--cohomology will not change by \autoref{lem:topological_invariance_of_etale_cohomology}, and $X_{\overline{k}}$ is again globally $F$--split by \cite[Lemma 2.4]{Gongyo_Li_Patakfalvi_Schwede_Tanaka_Zong_On_rational_connectedness_of_globally_F_regular_threefolds}. 
    
    We can then argue as follows: since the action of the Frobenius on $H^i(X, \cO_X)$ is injective by the global $F$--splitness assumption, we obtain by \cite[Corollary p.143]{Mumford_Abelian_varieties} that $H^i(X, \cO_X)$ has a $k$--basis of elements fixed by $\phi$. This finishes the proof.
\end{proof}

\begin{remark}\label{we_suck_at_writing}
    Note that the proof of the fact that \[ H^i_{\et}(X, \bF_p) \cong \set{x \in H^i(X, \cO_X)}{\phi(x) = x} \inc H^i(X, \cO_X) \] did not use that $X$ was globally $F$--split. In particular, this shows that for any morphism $f \colon Y \to X$ of proper $k$--schemes, if $f$ induces an isomorphism $H^i(X, \cO_X) \to H^i(Y, \cO_Y)$ for some $i \geq 0$, then it also induces an isomorphism $H^i(X, \bF_p) \to H^i(Y, \bF_p)$. It follows formally that the same holds for $\bF_p$ replaced by $\bQ_p$.
\end{remark}

The following result is the main trick that allows us to prove the main theorem for fourfolds.
\begin{lemma}[{\cite[Proposition 8.4]{Chambert_Loir_Points_Rationnels_et_Groupes_Fondamentaux}}]\label{lem:euler_char}
    Let $X$ be a proper scheme over an algebraically closed field $k$. Then \[ \chi_{\et}(X, \bF_p) = \chi_{\et}(X, \bQ_p). \]
\end{lemma}
\begin{proof}
    Let us explain carefully the argument in \cite{Chambert_Loir_Points_Rationnels_et_Groupes_Fondamentaux}. By the Artin--Schreier sequence, each $H^i_{\et}(X, \bF_p)$ has finite dimension and $H^i_{\et}(X, \bF_p) = 0$ for $i > \dim(X)$. Now, we know from \stacksproj{0F0F} that for all $n \geq 0$, we have a natural isomorphism \[ \RGamma_{\et}(X, \bZ/p^{n + 1}\bZ) \otimes^L \bZ/p^n\bZ \cong \RGamma_{\et}(X, \bZ/p^n\bZ). \] Thus, we deduce from \stacksproj{0CQF} that each $H^i_{\et}(X, \bZ_p)$ has finite rank, and \[ \RGamma_{\et}(X, \bZ_p) \otimes^L \bF_p \cong \RGamma_{\et}(X, \bF_p). \] Since $\bZ_p$ is a regular ring, we know by \stacksproj{066Z} that there exists a bounded complex $V^{\bullet}$ of finite free $\bZ_p$--modules which is quasi--isomorphic to $\RGamma_{\et}(X, \bZ_p)$, so \[ V^{\bullet} \otimes_{\bZ_p} \bF_p \cong \RGamma_{\et}(X, \bF_p) \mbox{ \: \: \: and \: \: \:} V^{\bullet} \otimes_{\bZ_p} \bQ_p \cong \RGamma_{\et}(X, \bQ_p). \] Thus, we deduce that \[ \chi_{\et}(X, \bF_p) = \sum_i (-1)^i\dim_{\bF_p}H^i_{\et}(X, \bF_p) = \sum_i (-1)^i\dim_{\bF_p}(V^i \otimes \bF_p) = \sum_i(-1)^i\rank_{\bZ_p}(V^i), \] and the same holds for $\chi_{\et}(X, \bQ_p)$.
\end{proof}

\begin{remark}
    Although we only stated the result for $\bF_p$ and $\bQ_p$, the same proof shows that for any constructible $\bZ_p$--sheaf $\cF$ on $X$ as above, we have $\chi_{\et}(X, \cF \otimes_{\bZ_p}^L \bF_p) = \chi_{\et}(X, \cF \otimes_{\bZ_p} \bQ_p)$.
\end{remark}

\begin{definition} \label{def: nilp_rat_sing}
    We say that a normal $\bF_p$--scheme $X$ has \emph{$\mathbb{F}_p$--rational singularities} (resp.\ \emph{$\mathbb{Q}_p$--rational singularities}) if there exists a projective resolution of singularities $\pi \colon Y \to X$ such that $R^i\pi_*\bF_p = 0$ (resp.\ $R^i\pi_*\bQ_p = 0$) for all $i > 0$.
\end{definition}

\begin{remark}
    As for $W\cO$--rational singularities, $\bQ_p$--rational singularities can be defined via regular alterations instead of resolution of singularities. However, in our case, we will need resolution of singularities for other purposes, so this is why we define it this way.
\end{remark}

\subsection{Dual complexes}
In this section, we define the dual complex $\cD(X)$ of a pure dimensional scheme $X$ (provided it satisfies $(\star)$, see \autoref{def:dual_complex_condition_star}). It is a topological space encoding the combinatorial data of how the irreducible components of $X$ intersect. As we will see, if arbitrary intersections of irreducible components have no higher cohomology, then the $\bQ_p$-- and $\bF_p$--cohomology groups of $X$ can be identified with the corresponding singular cohomology groups of the dual complex.

\begin{definition}\label{def:dual_complex_condition_star}
    Let $X = \bigcup_{i \in [n]} X_i$ be a pure dimensional scheme, with $n$ irreducible components $X_i$. For a subset $J\subseteq [n]$, denote $X_J \coloneqq \bigcap_{j \in J} X_j$.  We say that $X$ satisfies the condition $(\star)$ if the following holds: for every $\emptyset\neq J \subseteq [n]$, if $X_J$ is non empty, then every connected component of $X_J$ is irreducible and has codimension $|J| - 1$ in $X$. In this case, an irreducible closed subset $Z\subseteq X$ is called a \emph{stratum of codimension $r$}, if it is an irreducible component of $X_J$ for some $J\subseteq [n]$ with $|J|=r+1$. Note that under the given assumption, $J$ is uniquely determined by the stratum $Z$, and so we denote it by $J(Z)\coloneqq J$.
\end{definition}

An immediate consequence of $(\star)$ is that $X_J$ is the disjoint union of its irreducible components. Therefore, for any stratum $Z$ of codimension $r$ and every $j \in J(Z)$, $Z$ lies in a unique irreducible component of $X_{J(Z) \setminus \{j\}}$, i.e.\ a stratum $Z'$ of codimension $r-1$ with $J(Z')=J(Z)\setminus\{j\}$. While these incidence relations don't necessarily define an abstract simplicial complex (because two distinct strata of codimension $r$ could lie in the same set of strata of codimension $r-1$), they do in fact define a $\Delta$--set (see Section 2 in \cite{Ranicki_Weiss_On_the_algebraic_L_theory_of_Delta-sets}) and thus a $\Delta$--complex (see \cite[Chapter 2.1]{Hatcher_algebraic_topology}).

\begin{definition}[{\cite[Definition 8]{De_Fernex_Kollar_Xu_The_dual_complex_of_sings}}]\label{def:dual_complex_pure_dim}
Let $X=\bigcup_{i\in[n]} X_i$ be a pure dimensional scheme satisfying $(\star)$. The \textit{dual complex} $\cD(X)$ of $X$ is the $\Delta$--complex constructed as follows:
\begin{itemize}
    \item The $r$--dimensional simplices of $\cD(X)$ are indexed by the strata of codimension $r$ of $X$.
    \item Let $\sigma_Z\subseteq\bR^{r+1}$ be the $r$--dimensional simplex corresponding to some stratum $Z$. The vertices of $\sigma_Z$ are identified with $J(Z)\subseteq[n]$ in an order preserving way.
    \item Let $\sigma_Z$ be an $r$--dimensional simplex corresponding to some stratum $Z$ and $J(Z)=\{j_0,\ldots,j_r\}$ with $j_0<\cdots<j_r$ the corresponding index set. For $0\leq i\leq r$ arbitrary, let $Z'$ be the unique stratum of codimension $r-1$ containing $Z$ such that $J(Z')=J(Z)\setminus\{j_i\}$. We then identify the face of $\sigma_Z$ whose vertices are $J(Z)\setminus\{j_i\}$ with the $(r-1)$--dimensional simplex $\sigma_{Z'}$, so that the vertices match.
\end{itemize}
\end{definition}

\begin{remark}
\begin{itemize}
    \item Note that $\cD(X)$ is the topological realization of the abstract $\Delta$--set which is described by the above incidence relations (see \cite[Definition 2.2]{Ranicki_Weiss_On_the_algebraic_L_theory_of_Delta-sets}). As the topological space $\cD(X)$ doesn't depend on the ordering of the irreducible components, we will mostly refrain from putting an arbitrary order on them, and just work with arbitrary index sets.
    \item Note that contrary to \cite[Definition 8]{De_Fernex_Kollar_Xu_The_dual_complex_of_sings}, we don't impose the irreducible components of $X_J$ to be normal. This is to include the case of low characteristics: the dual complex which will be of interest for us is $\cD(E)$ for some threefold dlt modification $(Y,E)$ as in \autoref{existence_dlt_modifications}, and we do not know whether the components of $E$ are normal (see \cite{Cascini_Tanaka_Purely_log_terminal_3folds_with_non_normal_centers_in_char_2, Bernasconi_Non-normal_purely_log_terminal_centres}). However, \autoref{prop:dlt_structure}.\autoref{prop:dlt_struct_geometry} still ensures $(\star)$ in this case of interest, and with the weak normalization explored in \autoref{subsection:MMP} we can bypass the potential issues arising from non-normality.
\end{itemize}
\end{remark}

\begin{lemma}\label{lem:dlt_dual_complex}
Let $(Y,E)$ be a dlt pair, where $E=\sum_i E_i$ is reduced and each $E_i$ is $\bQ$--Cartier. Then $E$ satisfies $(\star)$. In particular, $\cD(E)$ is well--defined.
\end{lemma}

\begin{proof}
    This is \autoref{prop:dlt_structure}.\autoref{prop:dlt_struct_geometry}.
\end{proof}

To relate $\bQ_p$-- and $\bF_p$--cohomology of $X$ with singular cohomology of $\cD(X)$, we need to ``resolve'' a sheaf on $X$ by its restrictions to the strata. More precisely, if $\cF$ is a sheaf on $X$, we make use of the complex
\begin{align}\label{ex_seq_et_sheaves}
    0\to\cF\to\bigoplus_i \iota_*\cF|_{X_i}\to\bigoplus_{i<j}\iota_*\cF|_{X_{ij}}\to\cdots\to\bigoplus_{\substack{J\subseteq[n] \\ |J|=r}}\iota_*\cF|_{X_J}\to\cdots
\end{align}
where for $J,J'\subseteq[n]$ with $|J|=r+1$ and $|J'|=r$ the map $\iota_*\cF|_{X_{J'}}\to \iota_*\cF|_{X_{J}}$ is $0$ if $J'\not\subseteq J$ and $(-1)^{i-1}$ times the restriction map if $J=\{j_0<\cdots<j_i<\cdots<j_r\}$ and $J'=J\setminus\{j_i\}$. One easily checks that this is a complex, but for arbitrary $\cF$ and $X$ it may not be exact. However, we will often use the following.
\begin{lemma}\label{lem:exact_dual_complex}
    Let $X$ be a scheme and let $X=\bigcup_{i\in[n]}X_i$ be any cover by closed subschemes. Then the sequences of étale sheaves
    \begin{align*}
        0\to\bF_{p,X}\to\bigoplus_i \bF_{p,X_i}\to\bigoplus_{i<j}\bF_{p,X_{ij}}\to\cdots\to\bigoplus_{\substack{J\subseteq[n] \\ |J|=r}}\bF_{p,X_J}\to\cdots
    \end{align*}
    and
    \begin{align*}
        0\to\bQ_{p,X}\to\bigoplus_i \bQ_{p,X_i}\to\bigoplus_{i<j}\bQ_{p,X_{ij}}\to\cdots\to\bigoplus_{\substack{J\subseteq[n] \\ |J|=r}}\bQ_{p,X_J}\to\cdots.
    \end{align*}
    as in \autoref{ex_seq_et_sheaves} are exact.
\end{lemma}
\begin{proof}
    Let us first show the statement for the sheaf $\bF_p$. By the same argument as in \cite[Corollary 2.3]{Berthelot_Bloch_Esnault_On_Witt_vector_cohomology_for_singular_varieties} we can reduce to the case $n=2$. In this case the statement is clear, as it can be checked e.g. at stalks.

    The same proof works for the étale sheaves $\bZ/p^n\bZ$. The statement for $\bQ_p$ is now a formal consequence.
\end{proof}

\begin{remark}
    \begin{itemize}
        \item The fact that we work with these sheaves specifically is crucial. If we wanted this to hold the sheaf $\cO_X$, we would need further restrictions on the cover (such as an snc condition).
        \item  By \cite[Corollary 2.3]{Berthelot_Bloch_Esnault_On_Witt_vector_cohomology_for_singular_varieties}, the analogous statement holds for the sheaf $W\cO_{X, \bQ}$.
    \end{itemize}
\end{remark}

The next lemma allows to compare the $\bQ_p$-- and $\bF_p$--cohomology of $X$ with the singular cohomology of $\cD(X)$ with coefficients in $\bQ_p$ resp.\ $\bF_p$, provided the strata of $X$ have vanishing higher cohomology.

\begin{lemma}\label{lem:cohomology_dual_complex_and_coherent}
Let $X = \bigcup_{i \in I} X_i$ be a pure dimensional scheme satisfying $(\star)$, proper over a separably closed field $k$ of characteristic $p>0$. Let $L\in\{\bQ_p,\bF_p\}$, and assume that for every stratum $Z\subseteq X$ and every $r>0$ we have $H^r_{\et}(Z,L)=0$.  Then \[ H^r_{\et}(X,L) \cong H^r(\cD(X),L) \] for every $r \geq 0$, where the group on the right--hand side denotes singular cohomology with coefficients in $L$.
\end{lemma}

\begin{proof}
    The complex
    \begin{align*}
        0\to L_{X}\to\bigoplus_{\substack{Z_0\text{ stratum}\\\codim Z_0=0}} L_{Z_0}\to\bigoplus_{\substack{Z_1\text{ stratum}\\\codim Z_1=1}} L_{Z_1}\to\cdots\to\bigoplus_{\substack{Z_r\text{ stratum}\\\codim Z_r=r}} L_{Z_r}\to\cdots,
    \end{align*}
    of \autoref{ex_seq_et_sheaves} is exact by \autoref{lem:exact_dual_complex} (recall that $X_J$ is the disjoint union of its irreducible components). We then apply the same proof as in \cite[Lemma 3.63]{Kollar_Singularities_of_the_minimal_model_program}, where the key is that the complex \[ \bigoplus_{\substack{Z_0\text{ stratum}\\\codim Z_0=0}} L_{Z_0}\to\bigoplus_{\substack{Z_1\text{ stratum}\\\codim Z_1=1}} L_{Z_1}\to\cdots\to\bigoplus_{\substack{Z_r\text{ stratum}\\\codim Z_r=r}} L_{Z_r}\to\cdots \] is both an acyclic resolution of $L_X$, and computes simplicial cohomology of the $\Delta$--complex $\cD(X)$ (see \cite[Chapter 2.1]{Hatcher_algebraic_topology}).
\end{proof}

Finally, we will need the following lemma later on, which gives a criterion for when a finite $\Delta$--complex $\cD$ is a real topological manifold of dimension $2$. To phrase it efficiently, we introduce a bit of notation: suppose that $\cD$ is of dimension $2$ (i.e.\ the maximal dimension of a simplex in $\cD$ is $2$). Then the underlying $\Delta$--set $\Sigma$ of $\cD$ can be encoded by a triple $\Sigma=(\Sigma_0,\Sigma_1,\Sigma_2)$, such that $\Sigma_0$ is the set of vertices, $\Sigma_1$ is the multiset of edges and $\Sigma_2$ is the multiset of $2$--simplices with sides in $\Sigma_1$. Note that as we are considering a $\Delta$--complex instead of a simplicial complex, there might be multiple edges between two vertices, or multiple $2$--simplices sharing the same triangle of edges (hence we work with multisets). Now to the criterion.

\begin{lemma}\label{lem:dual_complex_2manifold}
    Let $\cD$ be a finite $\Delta$--complex of dimension $2$. Denote by $\Sigma=(\Sigma_0,\Sigma_1,\Sigma_2)$ the corresponding $\Delta$--set. To any vertex $v\in \Sigma_0$, we associate a graph $G_v$ as follows: the vertex set $V(G_v)$ is given by the multiset of edges in $\Sigma_1$ containing $v$. For distinct vertices $e,f\in V(G_v)$, we draw an edge between $e$ and $f$ for every $2$--simplex $t\in \Sigma_2$ having $e$ and $f$ as sides (i.e.\ multiple such $2$--simplices give multiple such edges). 
    
    Suppose now that for every $v\in V$, $G_v$ is a cycle of length at least $2$. Then $\cD$ is a real topological manifold of dimension $2$.
\end{lemma}

\begin{proof}
    Clearly $\cD$ is Hausdorff and second countable, so we are left to prove that $\cD$ is locally $2$--Euclidean. Clearly $\cD$ is locally $2$--Euclidean around points in the interior of a triangle. Now for any edge $e$ of $\cD$, observe that $\cD$ is locally $2$--Euclidean around points in the interior of $e$ if and only if $e$ is contained in exactly two $2$--simplices. Let $v\in e$ be an endpoint of $e$. Notice that the $2$--simplices containing $e$ are precisely the edges of $G_v$ containing $e\in V(G_v)$. As $G_v$ is a cycle of length at least two, $e$ has degree $2$ in $G_v$, and hence is contained in precisely two $2$--simplices. Finally, for vertices $v\in V$, the fact that $G_v$ is a cycle of length $n\geq 2$ gives that locally around $v$, $\cD$ is homeomorphic to an $n$--gon. In particular, it is locally $2$--Euclidean around $v$.
\end{proof}

\section{The case of $\bF_p$--rational singularities}\label{sec:Fp-rat}

As stated in the introduction, the main theorem is pretty easy in the case of $\bF_p$--rational singularities. As it turns out, already $\bF_p$--pseudorational singularities (which do not need resolution of singularities to be defined) are enough for us. We point out that this section is independent to the rest of the paper.

Let us first recall what the \emph{Cartier operator} is. Let $k$ be an $F$--finite field, and fix an isomorphism $\cO_{\Spec k} \cong F^!\cO_{\Spec k}$ (note that if $k$ is perfect, there is a canonical choice since $F^! \cong F^*$ in this case). Now, for any $k$--variety $f \colon Z \to \Spec k$, we can apply $f^!$ to this fixed isomorphism and obtain an isomorphism $\omega_Z^{\bullet} \to F^!\omega_Z^{\bullet}$. By adjunction (since $F$ is finite, $F^!$ is simply the right adjoint of $F_*$), we obtain a morphism $F_*\omega_Z^{\bullet} \to \omega_Z^{\bullet}$. Taking smallest cohomology sheaves gives us the Cartier operator \[ \Tr_Z \colon F_*\omega_Z \to \omega_Z.\]

Now, recall that for any generically finite surjective morphism $f \colon Y \to X$ of $k$--varieties, we have an induced trace map $f_*\omega_Y \to \omega_X$. Indeed, we can apply Grothendieck duality to the natural map $f^* \colon \cO_X \to Rf_*\cO_Y$, to obtain $Rf_*\omega_Y^{\bullet} \to \omega_X^{\bullet}$. Taking cohomology sheaves in smallest degree gives the desired morphism $f_*\omega_X \to \omega_Y$.

Since $f$ commutes with the Frobenius, we obtain by abstract nonsense that the diagram 
\begin{equation*}
    \begin{tikzcd}
    f_*F^e_*\omega_Y = F^e_*f_*\omega_Y  \arrow[rr] \arrow[d, "f_*\Tr^e_Y"] &  & F^e_*\omega_X \arrow[d, "\Tr^e_X"] \\
f_*\omega_Y \arrow[rr]                                &  & \omega_X             
\end{tikzcd} \end{equation*} also commutes.

\begin{definition}
    Let $X$ be a normal variety. 
    \begin{itemize}
        \item We say that $X$ has \emph{weakly pseudorational singularities} if for any proper birational morphism $\pi \colon Y \to X$ with $Y$ normal, the inclusion $\pi_*\omega_Y \to \omega_X$ is an equality.
        \item We say that $X$ has \emph{$\bF_p$--pseudorational singularities} if for every proper birational morphism $\pi \colon Y \to X$ with $Y$ normal, there exists $e > 0$ such that the image of $\Tr^e_X$ lies in $\pi_*\omega_Y$.
    \end{itemize}
\end{definition}

\begin{remark}
    One may be confused, as for the appearance of the word weakly in ``weakly pseudorational singularities'' but not in ``$\bF_p$--pseudorational singularities'', while the definitions are in the same spirit. The reason is that usually, ``pseudorational singularities'' also imply Cohen--Macaulayness by definition. However, we do not need Cohen--Macaulayness for our purposes.
    
    More importantly, the main results of \cite{Baudin_Bernasconi_Kawakami_Frobenius_GR_fails} seem to show that the notion of $\bF_p$--pseudorational singularities (or $\bF_p$--rational) may be more natural than their version where we further impose some sort of Cohen--Macaulayness assumption.
\end{remark}

It is immediate that weakly pseudorational singularities are $\bF_p$--pseudorational. Here is a partial converse: 

\begin{lemma}\label{lem:F-pure_F_p-pseudorat_implies_pseudorat}
    An $F$--pure, $\bF_p$--pseudorational singularity is weakly pseudorational.
\end{lemma}
\begin{proof}
    Let $X$ be a normal variety and let $\pi \colon Y \to X$ be a proper birational map, where $Y$ is also normal. Then by definition, for all $e \gg 0$, the image of $\Tr^e_X \colon F^e_*\omega_X \to \omega_X$ lies in $\pi_*\omega_Y$. To conclude, it is then enough to show that $\Tr^e_X$ is surjective. Since by definition, the Frobenius map $\cO_X \to F^e_*\cO_X$ locally splits, we obtain from Grothendieck duality that also $F^e_*\omega_X \to \omega_X$ locally splits. In particular, it is surjective.
\end{proof}

\begin{lemma}\label{lem:F_p-rat_implies_F_p-pseudorat}
    $\bF_p$--rational singularities are $\bF_p$--pseudorational.
\end{lemma}
\begin{proof}
    See \cite[Remark 3.5.(e)]{Baudin_Bernasconi_Kawakami_Frobenius_GR_fails}.
\end{proof}

The following is classical (see e.g. \cite[Corollary 5.24]{Kollar_Mori_Birational_geometry_of_algebraic_varieties} or \cite[Proposition 12.5]{Krah_Vial_Proper_spliters_in_positive_characteristic} in great generality), but we recall the argument for convenience of the reader.
\begin{lemma}\label{main_thm:classical_statement}
    Let $X$ be a quasi--Gorenstein normal variety with weakly pseudorational singularities. Then $X$ has canonical singularities.
\end{lemma}
\begin{proof}
    Let $\pi \colon Y \to X$ be a proper birational map, with $Y$ a normal variety. Write \[ K_Y \sim \pi^*K_X + E_+ - E_-, \] where both $E_+$ and $E_-$ are effective. Assume by contradiction that $E_- \neq 0$. We then claim that $\pi_*\cO_Y(E_+ - E_-)$ is a strict ideal of $\cO_X$. Assuming the claim, we obtain that \[ \omega_X = \pi_*\omega_Y = \omega_X \otimes \pi_*\cO_Y(E_+ - E_-) \subsetneq \omega_X, \] hence giving a contradiction.

    Let us show the claim now. First, note that we have inclusions $\cO_X = \pi_*\cO_Y \subseteq \pi_*\cO_Y(E^+) \subseteq \left(\pi_*(\cO_Y(E^+)\right)^{\vee\vee} = \cO_X$, so we must have $\pi_*\cO_Y(E^+) = \cO_X$. On the other hand, the inclusion $\pi_*\cO_Y(E^+ - E^-) \subseteq \pi_*\cO_Y(E^+)$ must be strict, since the section $1$ lives in $\pi_*\cO_Y(E^+)$ but not in $\pi_*\cO_Y(E^+ - E^-)$. Hence, the claim is proven, and the proof is finished.
\end{proof}

The following is the main statement of this section.

\begin{proposition}[\autoref{prop_intro}]\label{main_thm_for_F_p-rational}
    Let $X$ be an $F$--pure, quasi--Gorenstein, normal variety with $\bF_p$--pseudorational singularities. Then $X$ has canonical singularities.

    In particular, this also holds if we replace ``$\bF_p$--pseudorational'' by ``$\bF_p$--rational''.
\end{proposition}
\begin{proof}
   To obtain the first statement, combine \autoref{main_thm:classical_statement} and \autoref{lem:F-pure_F_p-pseudorat_implies_pseudorat}. To conclude the part about $\bF_p$--rational singularities, use \autoref{lem:F_p-rat_implies_F_p-pseudorat}.
\end{proof}

\begin{remark}
    \begin{itemize}
        \item The $F$--pure hypothesis is crucial, as taking a cone over a supersingular elliptic curve shows (the vertex point is then a strictly log canonical, $\bF_p$--rational and Gorenstein singularity).
        \item Looking at the proof, one may think that we used less than $F$--pure, namely that the trace map $F_*\omega_X \to \omega_X$ is surjective. However, $\omega_X$ is a line bundle, so this surjectivity is in fact equivalent to $F$--purity.
    \end{itemize}
\end{remark}

As a corollary of \autoref{main_thm_for_F_p-rational}, we obtain an analogue of \autoref{Theorem_B} and \autoref{Theorem_C} in any dimension and characteristic, under the additional hypothesis that $X$ is Cohen--Macaulay (and therefore Gorenstein).

\begin{corollary}\label{cor:main_thm_under_F_p-CM}
    Let $X$ be an $F$--pure, Gorenstein, normal variety with $\bQ_p$--rational singularities over a perfect field. Then $X$ has canonical singularities.
\end{corollary}

\begin{proof}
    By \cite[Proposition 5.2.3]{Baudin_Grauert-Riemenschneider_vanishing_for_Witt_canonical_sheaves}, Cohen--Macaulayness and $\bQ_p$--rationality together imply that $X$ has $\bF_p$--pseudorational singularities. We then immediately conclude by \autoref{main_thm_for_F_p-rational}.
\end{proof}

\section{The case of surfaces}\label{sec:surfaces}

Here, we show the analogue of our main theorems in dimension two.

\begin{proposition}\label{main_thm_surfaces}
Let $X$ be a normal, quasi--Gorenstein, $F$--pure surface with $\bQ_p$--rational singularities. Then $X$ has canonical singularities.
\end{proposition}

\begin{proof}
    Up to shrinking $X$, we may assume that this variety has only one singular point $x \in X$. Let $\pi \colon Y \to X$ denote a minimal resolution, with exceptional divisor $E$. Then we have \[ K_Y + E \sim \pi^*K_X, \] so given that $X$ is $F$--pure, we can lift the $F$--splitting section to $Y$ and obtain that the pair $(Y, E)$ is globally sharply $F$--split (see \cite[Lemma 3.3]{Gongyo_Takagi_Surfaces_of_gloally_F_regular_and_F_split_type}). As in the proof of \autoref{cor:full_exc_div_GFS_and_K_trivial} (i.e.\ use adjunction and $F$--adjunction), we deduce that $E$ is globally $F$--split and satisfies $\omega_E \cong \cO_E$. Let $k \coloneqq H^0(E, \cO_E)$. Then by \cite[Lemma 2.4]{Gongyo_Li_Patakfalvi_Schwede_Tanaka_Zong_On_rational_connectedness_of_globally_F_regular_threefolds}, $E_{\overline{k}}$ is also globally $F$--split and satisfies $K_{E_{\overline{k}}} \sim 0$. In particular, $H^1(E_{\overline{k}}, \cO_{E_{\overline{k}}}) \neq 0$ by Serre duality, and the Frobenius action on this cohomology group is bijective. By \autoref{lem:wo_o_cohomology_ordinary}, we deduce that 
    \begin{equation}\label{surf_non_van}
        H^1_{\et}(E_{\overline{k}}, \bQ_p) \neq 0.
    \end{equation}
    
    On the other hand, using the $\bQ_p$--rationality condition and \autoref{lem:proper_base_change}, we obtain that $R^1\pi_{E, *}\bQ_p = 0$, where $\pi_E \colon E \to \Spec k$ is the induced map. Base changing to $\overline{k}$ and again using \autoref{lem:proper_base_change}, we obtain that $R^1\pi_{E_{\overline{k}}, *}\bQ_p = 0$, so in other words $H^1_{\et}(E_{\overline{k}}, \bQ_p) = 0$ (see \autoref{warning:etale_cohomology}). This is a contradiction with \autoref{surf_non_van}, so the proof is complete.
\end{proof}



\section{The case of threefolds}

Let $X$ be a normal threefold satisfying the usual hypotheses, i.e.\ $X$ is quasi--Gorenstein, $F$--pure and $\bQ_p$--rational. Having established the surface case reduces the threefold case to excluding isolated strictly lc singularities. 

In the surface case, we took a minimal resolution in order to prove our result. From this point of view, it is then natural to consider the dlt modification $(Y,E)$ of $X$ (see \autoref{existence_dlt_modifications}). As in the surface case, it will be crucial for us to understand the $\bQ_p$--cohomology of $E$ to derive a contradiction (it turns out that for threefolds, we will also need to understand its $\bF_p$--cohomology). To do so, we have to both understand the cohomology of the components $E_i$, together with the global geometry of $\cD(E)$.

\subsection{Description of the individual components}\label{ss:log_calabi_yau_surfaces}

Let us try understand the geometry and cohomology of the individual components $E_i$. By \autoref{prop:dlt_structure}.\autoref{prop:dlt_structure_adjunction} and $F$--adjunction for $K_Y+E\sim\pi^*K_X$, we obtain that each $(E_i^{\nu},\Delta_i)$ (where $\Delta_i=(E -E_i)|_{E_i^{\nu}}$) is a globally sharply $F$--split dlt pair with $K_{E_i^{\nu}}+\Delta_i\sim 0$. In particular, $(E^{\nu}_i,\Delta_i)$ is a dlt log Calabi--Yau surface. The idea of this section is as follows: 
\begin{itemize}
    \item show that we can base change our dlt pairs above to the algebraic closure, and even take a minimal resolution without losing relevant information (\autoref{lem:dlt_and_F_pure_surf_is_geom_dlt}, \autoref{lem:dlt_blowup_gorenstein_comps_of_exc});
    \item classify smooth dlt log Calabi--Yau pairs over an algebraically closed field (\autoref{lem:classification_dlt_surfaces}).
\end{itemize}

\begin{lemma}\label{lem:dlt_and_F_pure_surf_is_geom_dlt}
    Let $(S, \Delta = \sum_i \Delta_i)$ be a proper surface pair over an $F$--finite field $k$, such that $K_S + \Delta$ is Cartier and each $\Delta_i$ has coefficient 1. Assume further that $(S, \Delta)$ is dlt and geometrically sharply $F$--pure. Then $(S, \Delta)$ is geometrically dlt.
\end{lemma}

\begin{remark}
    Most likely, one can prove much more general statements than this one. We leave this for future work.
\end{remark}

\begin{proof}
    Up to base change, we may assume that $k$ is separably closed. Since $S$ is in particular geometrically reduced, we obtain that $H^0(S, \cO_S) = k$. Since $(S_{\overline{k}}, \Delta_{\overline{k}})$ is sharply $F$--pure, $S_{\overline{k}}$ is normal (see \autoref{lem:normal_geom_F_pure_is_geom_normal}) and the coefficients of $\Delta_{\overline{k}}$ must all be $1$. In particular, this shows that $H^0(\Delta_i, \cO_{\Delta_i}) = k$ too (if $H^0(\Delta_i, \cO_{\Delta_i})$ was a purely inseparable extension of $k$, then $\Delta_{i, \overline{k}}$ would have higher multiplicity as a divisor). 
    
    Note that for all $i$, $K_S + \Delta_i$ is Cartier. Indeed, it is Cartier outside of the other $\Delta_j$'s by assumption, and it is Cartier at intersection with other $\Delta_j$'s by the snc assumption. Since the pair $(S_{\overline{k}}, \Delta_{\overline{k}})$ is log canonical, we know by \cite[Theorem 2.31]{Kollar_Singularities_of_the_minimal_model_program} that each $\Delta_{i, \overline{k}}$ has at worst nodal singularities. In particular, each $\Delta_{i, \overline{k}}$ is $F$--pure, so they are smooth by \autoref{lem:normal_geom_F_pure_is_geom_normal}.
    
    Let $x \in S$, and let us show that $(S_{\overline{k}}, \Delta_{\overline{k}})$ is dlt at $x$ (recall that the map $S_{\overline{k}} \to S$ is a homeomorphism, hence the identification).
    
    \begin{itemize}
        \item If $x \notin \Delta$, then let us show that $S_{\overline{k}}$ is canonical at $x$. Since $S$ is rational at $x$ (\cite[Proposition 2.28]{Kollar_Singularities_of_the_minimal_model_program}), it is in particular $\bQ_p$--rational at $x$. Hence, so is $S_{\overline{k}}$ by \cite[Proposition 3.22]{Patakfalvi_Zdanowicz_Ordinary} (their proof shows that $\bQ_p$--rationality is a geometric property). Since $K_S$ is Cartier at $x$ by assumption, also $K_{S_{\overline{k}}}$ is Cartier at $x$. Thus, we deduce by \autoref{main_thm_surfaces} that $S_{\overline{k}}$ is in fact canonical at $x$. 
        
        \item If $x \in \Delta_i \cap \Delta_j$, we deduce by the dlt assumption ($(S, \Delta)$ is snc around $x$) and $F$--pure adjunction (see \autoref{rem:F_adj}) that locally around $x$, $(\Delta_i, \Delta_j|_{\Delta_i} = x)$ is sharply $F$--pure. We then deduce as we did at the beginning of the proof that $k(x) = k$. By the dlt assumption, $(S, \Delta)$ is snc around $x$, so we conclude that both $\Delta_i$ and $\Delta_j$ are Cartier at $x$, and that $(\Delta_i \cdot \Delta_j)_x = 1$ (we mean the intersection number at $x$). In particular, $\Delta_{i, \overline{k}}$ and $\Delta_{j, \overline{k}}$ are also Cartier at $x$, and $(\Delta_{i, \overline{k}} \cdot \Delta_{j, \overline{k}})_x = 1$. This forces $x$ to be a smooth point of $S_{\overline{k}}$, so $(S_{\overline{k}}, \Delta_{\overline{k}})$ is dlt at $x$.
        
        \item If $x \in \Delta_i$ and not in the other components of $\Delta$, then let us show that $(S_{\overline{k}}, \Delta_{\overline{k}})$ is plt at $x$. Since $(S, \Delta_i)$ is plt at $x$, we know that we have \[ (K_S + \Delta_i)|_{\Delta_i} \sim K_{\Delta_i} \] at $x$. Indeed, if there was a different, then $x$ would be an lc center of $(S, \Delta_i)$ by \cite{Kollar_Singularities_of_the_minimal_model_program}. The dlt assumption would then imply that both $K_S$ and $\Delta_i$ would be Cartier and $S$ would be smooth at $x$, contradicting the existence of a non--zero different.

        But then, by base changing we also have \[ (K_{S_{\overline{k}}} + \Delta_{i, \overline{k}})|_{\Delta_{i, \overline{k}}} \sim K_{\Delta_{i, \overline{k}}} \] at $x$. Since $\Delta_{i, \overline{k}}$ is smooth, we conclude the proof by plt inversion of adjunction (\cite[Theorem 4.9]{Kollar_Singularities_of_the_minimal_model_program}). \qedhere
    \end{itemize}
\end{proof}

\begin{remark}
    To conclude that the components $\Delta_{i, \overline{k}}$ are $F$--pure, we could have also used the more general version of $F$--adjunction \cite[Theorem B]{Polstra_Simpson_Tucker_On_F_pure_inversion_of_adjunction}.
\end{remark}

\begin{lemma}\label{lem:normal_geom_F_pure_is_geom_normal}
    Let $S$ be a normal and geometrically $F$--pure scheme of finite type over an $F$--finite field $k$. Then $S$ is geometrically normal.
\end{lemma}
\begin{proof}
    We can show the statement locally, so let $U = \Spec R$ be an affine open of $S$. By \cite[Lemma 2.2]{Tanaka_Behavior_of_canonical_divisors_under_purely_inseparable_base_changes}, the morphism $R_{\overline{k}} \to (R_{\overline{k}})^{\nu}$ is a universal homeomorphism. By \stacksproj{0CNF}, we then have a factorization
    \[ \begin{tikzcd}
		R_{\overline{k}} \arrow[r, "\nu"] \arrow[rr, "F^e"', bend right] & (R_{\overline{k}})^{\nu} \arrow[r] & R_{\overline{k}}
	\end{tikzcd} \] for some $e > 0$ big enough. Let $\psi \colon F^e_*R_{\overline{k}} \to R_{\overline{k}}$ be a splitting of the Frobenius. Then the composition
	\[ \begin{tikzcd}
		R_{\overline{k}} \arrow[r, "\inc"] & (R_{\overline{k}})^{\nu} \arrow[r, "(\cdot)^{p^e}"] & F^e_*R_{\overline{k}} \arrow[r, "\psi"] & R_{\overline{k}},
	\end{tikzcd} \] is the identity, so in other words the inclusion $R_{\overline{k}} \inc (R_{\overline{k}})^{\nu}$ splits as $R_{\overline{k}}$--modules. Since this inclusion is generically an isomorphism and $(R_{\overline{k}})^{\nu}$ is a torsion--free $R_{\overline{k}}$--module, we must have $R_{\overline{k}} = (R_{\overline{k}})^{\nu}$. In other words, $R_{\overline{k}}$ is normal.
\end{proof}

\begin{remark}
    In the following proposition, the components $E_i^{\nu}$ will be geometrically normal by \autoref{lem:normal_geom_F_pure_is_geom_normal}. Thus, the notation $E_{i, \overline{k}}^{\nu}$ will not lead to any confusion.
\end{remark}

\begin{definition}
    A pair $(T, D)$ is said to be smooth dlt log Calabi--Yau if $T$ is smooth, proper over an algebraically closed field, $D$ is a reduced divisor, the pair $(T, D)$ is dlt and satisfies $K_T + D \sim 0$.
\end{definition}

\begin{lemma}\label{lem:dlt_blowup_gorenstein_comps_of_exc}
 Let $\pi \colon (Y,E) \to X$ be the dlt modification of a quasi--Gorenstein, $F$--pure threefold singularity which is canonical in codimension two, such that the residue field of $X$ is separably closed and $F$--finite. Let $E=\bigcup_i E_i$ be the decomposition into irreducible components, let $\Delta_i=(E-E_i)|_{E_i^{\nu}}$ be the associated different divisor on $E_i^{\nu}$, and let $k \coloneqq H^0(E^{\wn}, \cO_{E^{\wn}})$. Then the following holds:
 
 \begin{enumerate}
     \item each $E_{i, \overline{k}}^{\nu}$ has canonical singularities, the minimal resolution $\pi_i\colon T_i\to E_{i, \overline{k}}^{\nu}$ is an isomorphism over $D_i\coloneqq\Delta_{i,\overline{k}}$, and $(T_i,D_i)$ is a smooth dlt log Calabi--Yau pair; 
     \item for any $L \in \{\bF_p, \bQ_p\}$ and $r 
     \geq 0$, the natural morphism \[ H^r_{\et}(E_i, L) \to H^r_{\et}(T_i, L) \] is an isomorphism.
 \end{enumerate} 
 \end{lemma}

\begin{proof}
    Since $k$ is a finite extension of a separably closed field, it is itself separably closed. By \autoref{cor:full_exc_div_GFS_and_K_trivial}, we know that $E^{\wn}$ is globally $F$--split, so it is geometrically globally $F$--split (over $k$) by \cite[Lemma 2.4]{Gongyo_Li_Patakfalvi_Schwede_Tanaka_Zong_On_rational_connectedness_of_globally_F_regular_threefolds}. In particular, it is geometrically reduced, so all $E_i^{\nu}$ are generically geometrically reduced too. This forces $H^0(E_i^{\nu}, \cO_{E_i^{\nu}}) = k$ (recall that $k$ is separably closed).

    By \autoref{cor:full_exc_div_GFS_and_K_trivial}.\autoref{F_adj_cpnts}, we also know that each $(E_i^{\nu}, \Delta_i)$ is a globally sharply $F$--split dlt pair with $K_{E_i^{\nu}} + \Delta_i \sim 0$ (in particular it is Cartier). By the same argument as in \cite[Corollary 2.8]{Gongyo_Li_Patakfalvi_Schwede_Tanaka_Zong_On_rational_connectedness_of_globally_F_regular_threefolds} in the $F$--split setup, we obtain that $(E_i^{\nu}, \Delta_i)$ is geometrically globally sharply $F$--split. We then deduce from \autoref{lem:dlt_and_F_pure_surf_is_geom_dlt} that each $(E_i^{\nu}, \Delta_i)$ is geometrically dlt. 

    Since $E_i^{\nu}$ has canonical singularities by \autoref{prop:dlt_structure}.\autoref{prop:dlt_structure_it5}, we know that $K_{E_i^{\nu}}$ is Cartier by \cite[Theorem 2.29]{Kollar_Singularities_of_the_minimal_model_program}, so $K_{E_{i, \overline{k}}^{\nu}}$ is also Cartier. In particular, each $E^{\nu}_{i, \overline{k}}$ has canonical singularities too.

    Now consider the minimal resolution $\pi_i\colon T_i \to E^{\nu}_{i, \overline{k}}$. Since $E^{\nu}_{i, \overline{k}}$ has canonical singularities, $\pi_i$ is crepant. Since $E^{\nu}_{i, \overline{k}}$ is regular along $D_i \coloneqq \Delta_{i, \overline{k}}$ by \autoref{prop:dlt_structure}.\autoref{prop:dlt_structure_it4}, $\pi_i$ is an isomorphism over $D_i$. The pair $(T_i,D_i)$ is then a smooth dlt log Calabi--Yau surface pair.

    Let us show the last part for $L = \bF_p$ (the case $L = \bQ_p$ will then follow formally). For any $r \geq 0$, we have \[ H^r_{\et}(E_i, \bF_p) \underset{\autoref{prooflem:dlt_blowup_gorenstein_comps_of_exc_1}}{\cong} H^r_{\et}(E_{i, \overline{k}}^{\nu}, \bF_p) \underset{\autoref{prooflem:dlt_blowup_gorenstein_comps_of_exc_2}}{\cong} H^r_{\et}(T_i, \bF_p).\]
    In the above equation, we used the following:

    \begin{enumerate}
        \item\label{prooflem:dlt_blowup_gorenstein_comps_of_exc_1} Since both $E_i^{\nu} \to E_i$ and $E_{i, \overline{k}}^{\nu} \to E_i^{\nu}$ are universal homeomorphisms (see \autoref{prop:dlt_structure}.\autoref{prop:dlt_struct_geometry} for the first morphism), the map $H^r_{\et}(E_i, \bF_p) \to H^r_{\et}(E_{i, \overline{k}}^{\nu}, \bF_p)$ is an isomorphism by \autoref{lem:topological_invariance_of_etale_cohomology};
        \item\label{prooflem:dlt_blowup_gorenstein_comps_of_exc_2} Since $E_{i, \overline{k}}^{\nu}$ has canonical singularities, we know by \cite[Proposition 2.28]{Kollar_Singularities_of_the_minimal_model_program} that it has rational singularities, so the natural map $H^r(E^{\nu}_{i, \overline{k}},\cO_{E^{\nu}_{i, \overline{k}}}) \to H^r(T_i, \cO_{T_i})$ is an isomorphism. By \autoref{we_suck_at_writing}, the same holds for $\bF_p$--cohomology.\qedhere
    \end{enumerate}
\end{proof}

We now come to the classification of smooth dlt log Calabi--Yau surface pairs over an algebraically closed field.

\begin{lemma}\label{lem:classification_dlt_surfaces}
Let $(T,D)$ be a smooth dlt log Calabi--Yau surface pair. Then one of the following conditions hold:
\begin{enumerate}
\item\label{lem:classification_dlt_surfaces_it_CY} $K_T \sim 0$ and $D = 0$;
\item\label{lem:classification_dlt_surfaces_it_rul} there is a projective birational morphism $T \to S$, where $S$ is ruled over an elliptic curve $C$, and $D$ is a disjoint sum of at most two elliptic curves. Furthermore:
\begin{enumerate}
    \item\label{lem:classification_dlt_surfaces_it_rul_red} if $D$ is reducible (say $D = D_{+} + D_{-}$) then both maps $D_{\pm} \to C$ are isomorphisms;
    \item\label{lem:classification_dlt_surfaces_it_rul_irr} if $D$ is irreducible, then the induced morphism $D \to C$ is a degree two isogeny. Furthermore, this can only happen in characteristic $2$;
    
\end{enumerate}
\item\label{lem:classification_dlt_surfaces_it_rat_ell} $T$ is rational and $D$ is an elliptic curve;
\item\label{lem:classification_dlt_surfaces_it_rat_cyc} $T$ is rational (i.e.\ birational to $\bP^2$) and $D$ is a single cycle of smooth rational curves.
\end{enumerate}
\end{lemma}

\begin{proof}
Taking a contraction of a sequence of $(-1)$--curves $\pi \colon T \to T_{\rm min}$ we obtain a smooth minimal surface $T_{\rm min}$. Let $D_{\rm min} \coloneqq \pi_*D$. Note that each contraction $\pi_n \colon (T_n, D_n) \to (T_{n + 1}, D_{n + 1})$ of the above sequence is log crepant (i.e.\ $K_{T_n} + D_n \sim \pi_n^*(K_{T_{n + 1}} + D_{n + 1})$). Indeed, we know by induction that $K_{T_n} + D_n \sim 0$, so if we write \[ K_{T_n} + D_n \sim \pi_n^*(K_{T_{n + 1}} + D_{n + 1}) + a\Exc(\pi_n) \] for some $a \in \bZ$, intersecting with the divisor $\Exc(\pi_n)$ yields $a = 0$.

Thus, \[ K_T + D \sim \pi^*(K_{T_{\rm min}} + D_{\rm min}),\] so in particular $K_{T_{\rm min}} + D_{\rm min} \sim 0$ and $(T_{\rm min},D_{\rm min})$ is log canonical. Log crepant blowups $\psi \colon (V',D_{V'}) \to (V,D_V)$ with $(V, D_V)$ a log canonical surface pair with $V$ smooth fall in three classes:
\begin{enumerate}[(a)]
\item a blowup in a point of intersection of components of $D_V$, in this case $D_{V'} = \psi^{-1}_*D +\Exc(\psi)$;
\item a blowup in a nodal self--intersection of a component of $D_V$, in this case $D_{V'} = \psi^{-1}_*D + \Exc(\psi)$;
\item a blowup in a smooth point of $D_V$; in this case $D_{V'} = \psi^{-1}_*D$.  
\end{enumerate}

In particular, if at any point in the proof it happens that $(T_{\rm min}, D_{\rm min})$ satisfies \autoref{lem:classification_dlt_surfaces_it_CY}, \autoref{lem:classification_dlt_surfaces_it_rul}, \autoref{lem:classification_dlt_surfaces_it_rat_ell} or \autoref{lem:classification_dlt_surfaces_it_rat_cyc}, we will be done. Let us now understand $(T_{\rm min}, D_{\rm min})$. There are three cases of $T_{\rm min}$ to consider: $K_{T_{\rm}}$ is nef, $T_{\rm min} \cong \bP^2$, or $T_{\rm min}$ is ruled over a curve. 

In the first case, then also $-D_{\rm min}$ is nef, so $D_{\rm min} = 0$. Hence, we are in case \autoref{lem:classification_dlt_surfaces_it_CY}. Second, let us assume that $T_{\rm min} \cong \bP^2$. The divisor $D_{\min}$ is therefore a section of $\cO_{\bP^2}(3)$ such that the respective pair is log canonical. It is a straightforward verification to see that the only possibilities for $D_{\rm min}$ are either a smooth elliptic curve, a nodal curve, three lines or a quadric and a line intersecting transversely. In all the cases except the nodal one, the pair $(T,D)$ falls into cases \autoref{lem:classification_dlt_surfaces_it_rat_ell} or \autoref{lem:classification_dlt_surfaces_it_rat_cyc}. If we are in the nodal case, note that since $(T, D)$ is dlt, the nodal point has to be blown up at some point (i.e.\ some log crepant blowup is of type (b)). Thus, after this blowup, we end up in case \autoref{lem:classification_dlt_surfaces_it_rat_cyc}. 

Let us treat the remaining case: $T_{\rm min}$ is a ruled surface over a curve $C$. Let $F$ denote a fiber, and let $D_{\rm min} = D^h_{\rm min} + D^v_{\rm min}$ be its decomposition into horizontal and vertical components. By adjunction on $F$, we have
\begin{equation}\label{classification_dlt_CY:intersection_horizontal_cpnts_and_fiber}
    D^h_{\rm min}.F = 2,
\end{equation}
so $D^h_{\rm min}$ consists of either two sections or a two--to--one multisection. Let $E$ be a component of $D_{\rm min}$, write $D_{\rm min} = D'_{\rm min} + E$. By \cite[Theorem 2.31]{Kollar_Singularities_of_the_minimal_model_program}, the singularities of $E$ are at worst nodal, so denote by $\delta$ the number of nodes of $E$. Then by adjunction,
\begin{equation}\label{classification_dlt_CY:adjunction_horizontal_components}
    0 = (K_{T_{\rm min}} + E + D'_{\min}).E = \deg K_{E^{\nu}} + 2\delta + D'_{\min}.E.
\end{equation}
In particular, $E^{\nu}$ is either $\bP^1$ or an elliptic curve. 

Assume that for some component $E$ of $D_{\min}$, $E^{\nu}$ is an elliptic curve. Then by \autoref{classification_dlt_CY:adjunction_horizontal_components}, $\delta = 0$ (i.e.\ $E$ is smooth) and $E$ is isolated in $D_{\min}$. Since $E$ has to be horizontal, we deduce that $D_{\min}^v = 0$. 

\begin{itemize}
    \item If $C$ is also an elliptic curve, then all components of $D_{\min} = D_{\min}^h$ must also be elliptic curves, so in fact $D_{\min}$ is a disjoint union of elliptic curves. We then deduce from \autoref{classification_dlt_CY:intersection_horizontal_cpnts_and_fiber} that we are in cases \autoref{lem:classification_dlt_surfaces_it_rul_red} or \autoref{lem:classification_dlt_surfaces_it_rul_irr}. It is then the content of \autoref{lem:leaf_only_char_2} that case \autoref{lem:classification_dlt_surfaces_it_rul_irr} can only happen in characteristic $2$.
    \item If $C = \bP^1$, then the morphism $E \to C$ cannot be bijective, i.e.\ $E.F \geq 2$. We then conclude by \autoref{classification_dlt_CY:intersection_horizontal_cpnts_and_fiber} that $D_{\min} = E$, so we are in case \autoref{lem:classification_dlt_surfaces_it_rat_ell}.
\end{itemize}

Finally, suppose that all horizontal components are rational. We will show that we are in case \autoref{lem:classification_dlt_surfaces_it_rat_cyc}. Let $E$ denote a horizontal component, and let us first assume that $E$ is singular. Let $x\in E$ be a singular point, that is, a node. Then as $E.F\geq\mult_x(E)\geq 2$, we have by  \autoref{classification_dlt_CY:intersection_horizontal_cpnts_and_fiber} that $D_{\rm min}^h=E$. If by contradiction we have $D_{\rm min}^v\neq 0$, then in \autoref{classification_dlt_CY:adjunction_horizontal_components} we have $D'_{\rm min}.E\geq 2$. As $\deg K_{E^\nu}=-2$, we obtain $\delta= 0$, which is impossible as $E$ is singular. Therefore, we obtain in fact $D_{\rm min}=E$. By \autoref{classification_dlt_CY:adjunction_horizontal_components} we then obtain that $\delta=1$, so $E$ has precisely one node. By the same argument as in the case where $T_{\rm min}\cong\bP^2$, this node has to be blown up at some point, so we are in case \autoref{lem:classification_dlt_surfaces_it_rat_cyc}.

Hence, we are left with the case where all the components of $D_{\rm min}$ are smooth and rational. Note that for any two distinct components $E_0, E_1$ and $x \in E_0 \cap E_1$, the multiplicity of the intersection of $E_0 \cup E_1$ at $x$ is $1$ ($E_0 \cup E_1$ is at worst nodal at $x$ by \cite[Theorem 2.31]{Kollar_Singularities_of_the_minimal_model_program}, so this follows from the definition of a nodal point, see \cite[1.41]{Kollar_Singularities_of_the_minimal_model_program}). Then, by \autoref{classification_dlt_CY:adjunction_horizontal_components} we obtain that $(D_{\rm min}-E).E=2$ for every component $E$ of $D_{\rm min}$. Therefore, as all the appearing intersection multiplicities are equal to $1$, the degree of every component $E$ in the dual graph of $D_{\rm min}$ is precisely equal to $2$, i.e.\ the dual graph is a disjoint union of cycles of length at least $2$. In order to conclude that we are indeed in case \autoref{lem:classification_dlt_surfaces_it_rat_cyc}, it remains to see that $D_{\rm min}$ is connected. By \autoref{classification_dlt_CY:intersection_horizontal_cpnts_and_fiber}, $D_{\rm min}$ cannot consist of fibers only, and furthermore, if $D_{\rm min}$ contains a fiber, then it is automatically connected. So we are reduced to the case where $D_{\rm min}=D_{\rm min}^h$. But in this case, \autoref{classification_dlt_CY:intersection_horizontal_cpnts_and_fiber} again implies that there can be at most two components, so $D_{\rm min}$ must again be connected. Hence we are indeed in case \autoref{lem:classification_dlt_surfaces_it_rat_cyc}.
\end{proof}

\begin{remark}\label{rem:algebro_geometric_Poincare}
    In particular, \autoref{lem:classification_dlt_surfaces} states that for a smooth dlt log Calabi--Yau surface pair $(T,D)$, the dual complex $\cD(D)$ is a finite quotient of a sphere $\bS^k$ with $k< \dim T$. This phenomenon is conjectured to hold for log Calabi--Yau pairs in any dimension: it is the algebro--geometric Poincaré conjecture (see \cite[Question 4]{Kollar_Xu_The_dual_complex_of_Calabi_Yau_pairs}). By \emph{loc.\ cit.}, the conjecture is known to hold up to dimension $4$, at least over $\bC$.
\end{remark}

\begin{lemma}\label{lem:leaf_only_char_2}
    Let $f \colon T\to C$ be a smooth projective surface ruled over an elliptic curve $C$, and let $D\subset T$ be an elliptic curve such that $\pi|_D$ is of degree $2$ and $K_T+D\sim 0$. Then the characteristic of the base field must be equal to $2$.
\end{lemma}

\begin{proof}
    Suppose by contradiction that we are in characteristic $p\neq 2$. As $K_T\sim -D$, we have a short exact sequence
    \begin{equation}\label{eq:char2_ses}
        0\to\omega_T\to\cO_T\to\cO_D\to 0.
    \end{equation}
    Let us examine what happens if we push forward \autoref{eq:char2_ses} along $f \colon T\to C$.

    First of all, note that as the fibers of $f$ are projective lines, we have $f_*\omega_T=0$ as well as $R^1f_*\cO_T=0$. Also, since $f$ is a fibration, we have $f_*\cO_T=\cO_C$. We then deduce by relative Grothendieck duality that
    \begin{equation*}
        R^1f_*\omega_T\cong\cO_C.
    \end{equation*}
    
    Finally, as $f|_D$ is a degree $2$ map between elliptic curves in characteristic $p\neq 2$, it is automatically étale. Now by \cite[Exercise IV.2.7]{Hartshorne_Algebraic_geometry}, we must have
    \begin{equation*}
        (f|_D)_*\cO_D\cong\cO_C\oplus\cL,
    \end{equation*}
    where $\cL$ is a non-trivial $2$--torsion line bundle on $C$.

    Putting all of this together, we obtain that pushing forward \autoref{eq:char2_ses} along $f$, gives the exact sequence
    \begin{align*}
        0\to \cO_C\to \cO_C\oplus\cL\to \cO_C\to 0.
    \end{align*}
    Taking determinants gives us an isomorphism $\cL\cong\cO_C$, leading to a contradiction. Hence we must be in characteristic $2$.
\end{proof}

\begin{example}
    In characteristic $p = 2$, the behaviour described in \autoref{lem:leaf_only_char_2} can happen. Let us give an example.

    Let $C$ be an ordinary elliptic curve, and let $F_{C/k} \colon C \to C^{(p)}$ denote its $k$-linearized Frobenius. Its dual map $V  \colon C^{(p)} \to C$ is then the Verschiebung, and it is étale (of degree $p = 2$) since $C$ is ordinary (see e.g. \cite[Proposition 2.3.2]{Hacon_Pat_GV_Characterization_Ordinary_AV}) . By definition of dual maps, the group morphism $V^* \colon \Pic(C) \to \Pic(C^{(p)})$ corresponds to $F_{C/k}$, and hence is injective. Throughout, we will use this without further mention.

    Consider $T \coloneqq \bP_{C}(V_*\cO_{C^{(p)}})$ with its natural morphism $f \colon T \to C$. Let us show that the map $\cO_{C^{(p)}} = V^*\cO_C \to V^*V_*\cO_{C^{(p)}}$ splits (the proof will also show that $V^*V_*\cO_{C^{(p)}} \cong \cO_{C^{(p)}}^{\oplus 2}$). By Fourier--Mukai theory (see \cite{Mukai_Duality_between_DX_and_D_hat_X}), this is equivalent to showing that the natural map $k(0) \to F_{C/k, *}F_{C/k}^*k(0)$ splits, where $k(0)$ denotes the skyscraper sheaf of $k$ at $0 \in C^{(p)}$. Since $F_{C/k}^*k(0)$ is $\fm_{C, 0}^p$--torsion by definition of the Frobenius, we have that $F_{C/k, *}F_{C/k}^*k(0)$ is $\fm_{C^{(p)}, 0}$--torsion. It is then a $k$--vector space, so the map above automatically splits, and $F_{C/k, *}F_{C/k}^*k(0) \cong k(0)^{\oplus 2}$. We then have a base change diagram
    \[ \begin{tikzcd}
       C^{(p)}   \times \bP^1 \cong T' \arrow[d, "f'"'] \arrow[rr, "g"] &  & T \arrow[d, "f"] \\
        C^{(p)} \arrow[rr, "V"']                                  &  & C,    
    \end{tikzcd} \] 
    where $T' \coloneqq \bP_C(V^*V_*\cO_{C^{(p)}})$. Let $S_0$ (resp.\ $S_0'$) denote the section of $f$ (resp.\ $f'$) coming from the inclusion $\cO_C \hookrightarrow V_*\cO_{C^{(p)}}$ (resp.\ its pullback by $V$). Let also $S_1'$ be a section coming from a splitting of $\cO_{C^{(p)}} \hookrightarrow V^*V_*\cO_{C^{(p)}} \cong \cO_{C^{(p)}}^{\oplus 2}$, so that by construction, $S_1' \neq S_0'$ but $S_1' \sim S_0'$. Finally, set $D \coloneqq g(S_1')$; we will show that $D$ is an elliptic curve, that $D \sim -K_T$ and that $f|_D \colon D \to C$ has degree $2$ (and is étale).

    Assume by contradiction that $f|_D \colon D \to C$ has degree $1$. Since $\cO_C$ is integrally closed, then $f|_D$ is an isomorphism. Hence, $D$ corresponds to some sub-line bundle $L \hookrightarrow V_*\cO_{C^{(p)}}$. But then $S_1' \to D$ has degree two, so $g^{-1}(D) = S_1'$. In other words, $V^*L$ corresponds to the section $S_1'$, so in particular $V^*L \cong \cO_{C^{(p)}}$. We then obtain that $L \cong \cO_C$. Since (up to scaling) there is only one inclusion $\cO_C \hookrightarrow V_*\cO_{C^{(p)}}$, we conclude that $S_1' = S_0'$, contradicting our construction.

    Thus, $D \to C$ has degree $2$, so $S_1' \to D$ has degree $1$. Let $G \coloneqq \ker(C^{(p)} \to C) \cong \bZ/2\bZ$, so that $G$ acts on $T'$ and by construction, $g$ corresponds to the quotient by $G$. Let $\sigma$ be the involution generating $G$. Then $f^{-1}(D) = S_1' \cup \sigma(S_1')$, and since $S_1' \to D$ has degree $1$, we deduce that $S_1' \neq \sigma(S_1')$. Since $S_0'$ is fixed by $\sigma$, we have that \[ \sigma(S_1') \sim \sigma(S_0') = S_0', \]  so $S_1'.\sigma(S_1') = (S_0')^2 = 0$. Thus, $\sigma(S_1') \cap S_1' = \emptyset$, so $S_1' \to D$ is an isomorphism.

    Finally, let us show that $D \sim -K_T$. By our calculation above, $g^*D \sim 2S_0' \sim -K_{T'}$. Since $g$ is étale, $g^*K_{T} \sim K_{T'}$, so we obtain that $g^*(K_T + D) \sim 0$. Write $K_T + D \sim a.S_0 + b.f^*R$ for $a, b \in \bZ$ and a divisor $R$ on $C$ (see \cite[Proposition V.2.3]{Hartshorne_Algebraic_geometry}). Since $g^*S_0 \sim S_0'$, we deduce from $g^*(K_T + D) \sim 0$ that $a = 0$ and $V^*R \sim 0$, so again $R \sim 0$. In other words, $K_T + D \sim 0$ and the proof is complete. 
    
    Note that the pair $(T, D)$ is globally sharply $F$--split. Indeed, given that $H^0(T, \omega_T) = 0$ ($T$ is ruled), the natural morphism $H^0(T, \omega_T(D)) \to H^0(D, \omega_D)$ is injective (hence an isomorphism, since $H^0(T, \omega_T(D)) \neq 0$). By ordinarity of $D$, the Cartier operator acts bijectively on $H^0(D, \omega_D)$, and hence also on $H^0(T, \omega_T(D))$. In other words, the trace map $F_*\omega_T(D) \to \omega_T(D)$ is surjective on global sections. Given that $\omega_T(D) \cong \cO_T$, we therefore obtain the existence of a splitting, i.e. $(T, D)$ is globally sharply $F$--split.
\end{example}

\subsection{Global description of the dual complex}

Now that we have a precise description of $(E_{i,\overline{k}}^{\nu},\Delta_{i,\overline{k}})$, we essentially know the shape of the dual complex $\cD(E)$ around the vertices $E_i$. Using this local description, we now deduce a global description of $\cD(E)$, pointing out what cases of \autoref{lem:classification_dlt_surfaces} can fit together.

\begin{proposition}\label{prop:dlt_blowup_gorenstein_dual_complex}
 With the same setting and notations as in \autoref{lem:dlt_blowup_gorenstein_comps_of_exc}, the dual complex $\cD(E)$ is defined by \autoref{lem:dlt_dual_complex} and we have the following case distinction:
\begin{enumerate}
\item \label{prop:dlt_blowup_gorenstein_dual_complex_it_k3} the dual complex $\cD(E)$ is a point, and $E^{\nu}_{\overline{k}}$ is an irreducible, globally $F$--split surface satisfying $K_{E^{\nu}_{\overline{k}}} \sim 0$. The minimal resolution $(T,D)$ of $E^{\nu}_{\overline{k}}$ falls into case \autoref{lem:classification_dlt_surfaces_it_CY} of \autoref{lem:classification_dlt_surfaces};

\item \label{prop:dlt_blowup_gorenstein_dual_complex_it_ec} the dual complex $\cD(E)$ is one--dimensional, it is either a cycle or a path graph, and each $D_i$ is a disjoint union of at most two ordinary elliptic curves. Furthermore, we have the following:
\begin{enumerate}
    \item\label{prop:dlt_blowup_gorenstein_dual_complex_it_ec_deg2} if $E_i$ is a vertex of degree $2$ in $\cD(E)$ (i.e.\ $D_i\coloneqq\Delta_{i,\overline{k}}$ is reducible) then $(T_i,D_i)$ falls into case \autoref{lem:classification_dlt_surfaces_it_rul_red} of \autoref{lem:classification_dlt_surfaces};
    \item\label{prop:dlt_blowup_gorenstein_dual_complex_it_ec_deg1} if $E_i$ is a vertex of degree $1$ in $\cD(E)$ (i.e.\ $D_i$ is irreducible) then $(T_i,D_i)$ falls into case \autoref{lem:classification_dlt_surfaces_it_rul_irr} or \autoref{lem:classification_dlt_surfaces_it_rat_ell} of \autoref{lem:classification_dlt_surfaces}. In particular, if $p\neq 2$, then only case \autoref{lem:classification_dlt_surfaces_it_rat_ell} can happen;
\end{enumerate}

\item \label{prop:dlt_blowup_gorenstein_dual_complex_it_rc} the dual complex $\cD(E)$ is two--dimensional, it is a closed compact connected topological manifold, every $E_{i,\overline{k}}^{\nu}$ is a birational to $\bP^2$, and every pair $(T_i,D_i)$ falls into case \autoref{lem:classification_dlt_surfaces_it_rat_cyc} of \autoref{lem:classification_dlt_surfaces}.
\end{enumerate}
\end{proposition}

\begin{proof}
    By definition, the shape of the dual complex $\cD(E)$ around a vertex $E_i$ is determined by $(E-E_i)|_{E_i}$: the edges containing $E_i$ correspond to the components of $(E-E_i)|_{E_i}$, and the $2$--simplices containing $E_i$ correspond to the components of intersections of components of $(E-E_i)|_{E_i}$. Hence, as $E_{i,\overline{k}}^{\nu}\to E_i$ is a universal homeomorphism, the local behaviour of $\cD(E)$ around $E_i$ is completely determined by the divisor $D_i \coloneqq \Delta_{i,\overline{k}}$ on $E_{i,\overline{k}}^{\nu}$.
    
    Recall that the minimal resolution $T_i\to E_{i,\overline{k}}^{\nu}$ is an isomorphism over $D_i$ by \autoref{lem:dlt_blowup_gorenstein_comps_of_exc}, and the pair $(T_i,D_i)$ satisfies the hypotheses of \autoref{lem:classification_dlt_surfaces} by \autoref{lem:dlt_blowup_gorenstein_comps_of_exc}. Therefore, $D_i$ is either empty, a cycle of rational curves, or a disjoint union of at most two elliptic curves. Note also that we already have a useful piece of global information about $\cD(E)$: by Zariski's main theorem, $E$ is connected, and thus $\cD(E)$ is connected as well. We now distinguish cases.
    \begin{enumerate}
    \itemsep 2mm
    \item Assume that for some $i$, we have that $D_i=0$. Then $E_i$ is disjoint from all other components so by connectedness of $E$, we know that $E = E_i$. Since $E_i^{\nu}$ is globally $F$--split by \autoref{cor:full_exc_div_GFS_and_K_trivial}, it is also geometrically globally $F$--split by \cite[Corollary 2.8]{Gongyo_Li_Patakfalvi_Schwede_Tanaka_Zong_On_rational_connectedness_of_globally_F_regular_threefolds}. Hence, we are in case \autoref{prop:dlt_blowup_gorenstein_dual_complex_it_k3} of the present proposition. 
    
    \item By the above, we may assume that $D_i\neq 0$ for all $i$. Since $E$ is connected and each $E^{\nu}_{i,\overline{k}} \to E_i$ is a universal homeomorphism (\autoref{cor:strata_of_dlt_locus_is_S_2_up_to_univ_homeo}), either all $D_i$'s are disjoint unions elliptic curves, or they are all cycles of smooth rational curves. Assume that we are in the first case, i.e.\ disjoint unions of elliptic curves.
    
    In particular, all triple intersections of the $E_i$'s are empty, so $\cD(E)$ is one--dimensional. Seeing $\cD(E)$ as a graph, it is then connected with vertices of degree at most $2$, and so either a cycle or a path graph. Finally, by $F$--adjunction (see \autoref{rem:F_adj}), every component of $D_i$ is globally $F$--split, and thus an ordinary elliptic curve. As the degree of the vertex $E_i$ in $\cD(E)$ is given by the number of irreducible components of $D_i$, \autoref{lem:classification_dlt_surfaces} gives that we are in case \autoref{prop:dlt_blowup_gorenstein_dual_complex_it_ec} of the present proposition.
    
    \item We are now left with the last case where all the components $E^{\nu}_{i,\overline{k}}$ are rational surfaces and each $D_i$ is a cycle of smooth rational curves. We are going to prove that we are in case \autoref{prop:dlt_blowup_gorenstein_dual_complex_it_rc}. Note that then $\cD(E)$ is $2$--dimensional: as $D_i$ is a cycle, it has at least two irreducible components with non--empty intersection. As explained in the first paragraph of the proof, their intersection corresponds to a $2$--simplex containing the vertex $E_i$. Then, as $\cD(E)$ is by definition a finite $\Delta$--complex (hence compact) and as $E$ (and thus $\cD(E)$) is connected, it suffices to prove that $\cD(E)$ is a topological manifold. But note that for any vertex $E_i$ of $\cD(E)$, the edge--graph $G_{E_i}$ around $E_i$ (see \autoref{lem:dual_complex_2manifold}) is precisely given by the dual graph of $(E-E_i)|_{E_i}$, as explained in the beginning of the proof. Since $(E-E_i)|_{E_i}$ is universally homeomorphic to $D_i$, this dual graph is a cycle of length at least $2$. So we conclude by \autoref{lem:dual_complex_2manifold} that $\cD(E)$ is a topological manifold.\qedhere
\end{enumerate}
\end{proof}
\begin{remark}\label{rem:dual_complex}
    \begin{itemize}
        \item Note that \autoref{prop:dlt_blowup_gorenstein_dual_complex} and \autoref{lem:dual_complex_2manifold} are instances of a more general local--to--global principle: if for every vertex $E_i$ the dual complex of $(E-E_i)|_{E_i}$ is a finite quotient of a sphere (see \autoref{rem:algebro_geometric_Poincare}), then $\cD(E)$ is an orbifold. If furthermore it is always a sphere itself, then $\cD(E)$ is a manifold (see \cite[Lemma 5.2.5]{Mauri_Mazzon_Stevenson_On_the_geometric_P=W_conjecture} for this circle of ideas). The case \autoref{prop:dlt_blowup_gorenstein_dual_complex}.\autoref{prop:dlt_blowup_gorenstein_dual_complex_it_ec} is then just an instance of the fact that the only compact connected one--dimensional orbifolds are the interval and $\bS^1$.
        \item Note that in the proof above, the only instances where the $F$--purity assumption was used was to reduce to the case of an algebraically closed field, and to argue about global $F$--splitness of the surface in \autoref{prop:dlt_blowup_gorenstein_dual_complex_it_k3}, or ordinarity of the elliptic curves in \autoref{prop:dlt_blowup_gorenstein_dual_complex_it_ec}.

        In particular, if $X$ was of finite type over an algebraically closed field to begin with, then the same statement would hold without the $F$--purity condition of $X$ (but of course without the global $F$--splitness or the ordinarity consequences discussed above).
    \end{itemize}
\end{remark}

\subsection{The proof}

We now have all the required ingredients to prove the main theorem announced in the introduction, in the case of threefolds. 

\begin{thm}[{\autoref{Theorem_B}}]\label{main_thm_threefolds}
Assume that $p \neq 2$. Then every normal, quasi--Gorenstein, $F$--pure threefold with $\bQ_p$--rational singularities essentially of finite type over an $F$--finite field has canonical singularities.
\end{thm}

\begin{proof}
    Let $X$ be a threefold essentially of finite type over an $F$--finite field, such that $X$ is $F$--pure, has $\bQ_p$--rational singularities and that $K_X$ is Cartier. We want to show that $X$ has only canonical singularities. Since all the notions involved are invariant under étale extensions and localizations, we may assume that $X$ is the spectrum of a strictly Henselian ring. We will proceed by contradiction, assuming that $X$ is not klt (this is equivalent to being non--canonical since $K_X$ is Cartier).
    
    First of all, $X$ is canonical in codimension two by \autoref{main_thm_surfaces}. Since $\dim(X)  = 3$ and $X$ is local, it can only have a unique non--klt center, namely the closed point $x \in X$. Let $\pi \colon Y \to X$ be a dlt modification as in \autoref{existence_dlt_modifications}, so that by the contradiction hypothesis, we have $K_Y + E \sim \pi^*K_X$ with $E = \Exc(\pi) = \pi^{-1}(x)$ (hence it is non--zero). As in \autoref{lem:dlt_blowup_gorenstein_comps_of_exc}, set $k \coloneqq H^0(E^{\wn}, \cO_{E^{\wn}})$, where we recall that $E^{\wn}$ is the weak normalization of $E$ (the field $k$ is then automatically separably closed).
    
    Thanks to the $\bQ_p$--rationality hypothesis, we know by \autoref{lem:proper_base_change} that for all $r > 0$, 
    \begin{equation}\label{threefold_Q_p_rationality}
        H^r_{\et}(E, \bQ_p) = 0.
    \end{equation}
    Before doing a case by case analysis, let us show that 
    \begin{equation}\label{eq:non_vanishing_H2}
        H^2_{\et}(E, \bF_p) \neq 0
    \end{equation} 
    (this will only be used in Case 4 of the proof). By \autoref{cor:full_exc_div_GFS_and_K_trivial}, $E^{\wn}$ is globally $F$--split and $\omega_{E^{\wn}} \cong \cO_{E^{\wn}}$. By Serre duality, we obtain that \[ H^2(E^{\wn}, \cO_{E^{\wn}}) \cong H^0(E^{\wn}, \omega_{E^{\wn}})^{\vee} \neq 0 \] so by \autoref{lem:baby_Riemann_Hilbert}, we deduce that $H^2_{\et}(E^{\wn}, \bF_p) \neq 0$. Since $E^{\wn} \to E$ is a universal homeomorphism (see \autoref{lem:weak_nor_univ_homeo}), also $H^2_{\et}(E, \bF_p) \neq 0$. \medskip
    
    We will now do a case by case analysis of what $E$ can be, according to \autoref{prop:dlt_blowup_gorenstein_dual_complex}. Let $\cD(E)$ denote the associated dual complex.\medskip

    \textbf{Case 1: the dual complex is a point.} Equivalently, we know by \emph{loc.\ cit.} that $E^{\nu}_{\overline{k}}$ is an irreducible globally $F$--split surface with rational singularities, such that $K_{E^{\nu}_{\overline{k}}} \sim 0$. Let $T \to E^{\nu}_{\overline{k}}$ be a minimal resolution of $E^{\nu}_{\overline{k}}$. Since it is crepant, $T$ also is globally $F$--split and satisfies $K_T \sim 0$. For any $r > 0$, we then have \[ H^r_{\et}(T, \bQ_p) \expl{=}{\autoref{lem:dlt_blowup_gorenstein_comps_of_exc}} H^r_{\et}(E, \bQ_p) \expl{=}{\autoref{threefold_Q_p_rationality}} 0. \]
    However, by the Bombieri--Mumford classification and the fact that $p \neq 2$ (\cite[Theorems 5 and 6 and the Proposition afterwards]{Bombieri_Mumford_Enriques_classification_of_surfaces_in_char_p_II}), the surface $T$ either satisfies $\Alb(T) \neq 0$ or it is an ordinary K3 surface. In the latter case, we immediately conclude by \autoref{lem:wo_o_cohomology_ordinary} that $H^2_{\et}(T, \bQ_p) \neq 0$, giving a contradiction.
    
    Assume now that $\Alb(T) \neq 0$. Note that we have an isomorphism \[ H^1_{\et}(\Alb(T), \bQ_p) \cong H^1_{\et}(T, \bQ_p). \] Indeed, there is such an isomorphism in crystalline cohomology by \cite[II.3.11.2]{Illusie_Complexe_de_de_Rham_Witt_et_cohomologie_cristalline}, so there is also an isomorphism in $\bQ_p$--cohomology by II.5.4 in \emph{loc.\ cit.}. Thus, we also conclude that $H^1_{\et}(T, \bQ_p) \neq 0$, since equation II.7.1.2 in \emph{loc.\ cit.} and the fact that $\Alb(T)$ is ordinary (see \cite[Theorem 1.2]{Ejiri_When_is_the_Albanese_morphism_an_algebraic_fiber_space_in_positive___characteristic?}) tell us that $\dim_{\bQ_p}H^1_{\et}(\Alb(T), \bQ_p) = \dim(\Alb(T))$. We then have a contradiction in this case too.  \medskip
    
    \textbf{Case 2: the dual complex is a cycle.} Let $n$ denote the number of irreducible components of $E$. Since $\cD(E)$ is a cycle, all $\Delta_{i, \overline{k}}$ must consist of disjoint unions of two elliptic curves. Hence, each non--empty intersection $E_{ij}$ is either irreducible if $n > 2$, or has two disjoint irreducible components if $n = 2$. Hence, we have \[ \bigoplus_{i < j} H^0_{\et}(E_{ij}, \bQ_{p, E_{ij}}) = \bQ_p^{\oplus n}. \]
    Consider the exact sequence \[ 0 \to \bQ_{p, E} \to \bigoplus_i \bQ_{p, E_i} \to \bigoplus_{i < j} \bQ_{p, E_{ij}} \to 0 \] given by \autoref{lem:exact_dual_complex} (since $\cD(E)$ has dimension one, there are no non--empty triple intersections). Since $H^1_{\et}(E, \bQ_p) = 0$, we have an induced exact sequence of cohomology groups

    \[ 0 \to \bQ_p \to \bQ_p^{\oplus n} \to \bQ_p^{\oplus n} \to 0, \] which is impossible. \medskip

    \textbf{Case 3: the dual complex is a path graph.} As before, let $n$ denote the number of irreducible components of $E$. As usual, let us write $D_i \coloneqq \Delta_{i, \overline{k}}$ for all $i$.
    
    By assumption, we can order the irreducible components $E_1, \dots, E_n$ where $E_1$ is only connected to $E_2$, $E_n$ is only connected to $E_{n - 1}$ and for all $1 < i < n$, $E_i$ is connected to both $E_{i - 1}$ and $E_{i + 1}$. As in \autoref{lem:dlt_blowup_gorenstein_comps_of_exc}, let $\pi_i\colon T_i\to E^{\nu}_{i,\overline{k}}$ be the minimal resolution, which we recall is an isomorphism over $D_i$. By \autoref{prop:dlt_blowup_gorenstein_dual_complex}.\autoref{prop:dlt_blowup_gorenstein_dual_complex_it_ec}, $(T_1,D_1)$ and $(T_n,D_n)$ fall into case \autoref{lem:classification_dlt_surfaces_it_rat_ell} of \autoref{lem:classification_dlt_surfaces} (as we are in characteristic $p \neq 2$, case \autoref{lem:classification_dlt_surfaces_it_rul_irr} is impossible). On the other hand, for $1<i<n$, the pair $(T_i,D_i)$ falls into case \autoref{lem:classification_dlt_surfaces_it_rul_red}. For any $1\leq i\leq n$, let $D_{i,1}$ (resp.\ $D_{i,-1}$) be the component of $D_i$ coming from the intersection $E_{i,i+1}$ (resp.\ $E_{i,i-1})$. In particular, we have $D_{1,-1}=0$ and $D_{n,1}=0$.

    In order to analyze the $\bQ_p$--cohomology of $E$, we compute it for all the strata, i.e.\ for all $E_i$ and all $E_{i,i+1}$. These computations fall into three cases.\medskip

    \textit{Cohomology of $E_i$ for $1<i<n$.}
    Recall that $(T_i,D_i)$ falls into case \autoref{lem:classification_dlt_surfaces_it_rul} of \autoref{lem:classification_dlt_surfaces}. Therefore, there exists a birational morphism $T_i\to S_i$ where $S_i$ is smooth projective and ruled over an elliptic curve $C_i$. By composing, we hence obtain a fibration $f_i\colon T_i\to C_i$ whose fibers are trees of smooth rational curves. In particular, the components of $D_i$ (which are ordinary elliptic curves) dominate $C_i$, so the elliptic curve $C_i$ is ordinary as well. 
    
    In summary, combining the following two facts
    \[ \begin{cases}
        (F1):\quad Rf_{i, *}\cO_{T_i} = \cO_{C_i}; \\
        (F2):\quad C_i \mbox{ is an ordinary elliptic curve, }
    \end{cases} \] 
    we deduce that \[ H^1_{\et}(E_i,\bQ_p)\expl{\xrightarrow{\makebox[2cm]{$\cong$}}}{\autoref{lem:dlt_blowup_gorenstein_comps_of_exc}}H^1_{\et}(T_i,\bQ_p) \expl{\xleftarrow{\makebox[2cm]{$\cong$}}}{(F1) and \autoref{we_suck_at_writing}} H^1_{\et}(C_i,\bQ_p)\expl{=}{(F2)}\bQ_p,\] and similarly \[H^2_{\et}(E_i,\bQ_p)=0.\]
    
    \textit{Cohomology of $E_i$ for $i\in\{1,n\}$.} As explained in the beginning of Case 3 (where we used that $p \neq 2$), in this case $T_i$ is rational, so we have $H^1(T_i, \cO_{T_i}) = H^2(T_i, \cO_{T_i}) = 0$. We then obtain that $H^r_{\et}(E_i,\bQ_p)=0$ for all $r>0$. \medskip
    
    \textit{Cohomology of $E_{i,i+1}$.} Note that for $1\leq i<n$, $E_{i,i+1}$ is universally homeomorphic to $D_{i,1}$, which is an ordinary elliptic curve. Therefore, we obtain that \[H^1_{\et}(E_{i,i+1},\bQ_p) \expl{=}{\autoref{lem:topological_invariance_of_etale_cohomology}} H^1_{\et}(D_{i,1},\bQ_p)=\bQ_p.\]
    
    Now we are ready to analyze the $\bQ_p$--cohomology of $E$. Consider the exact sequence 
    \begin{align}\label{eq:SES_dual_comp}
        0 \to \bQ_{p,E} \to \bigoplus_{i} \bQ_{p,E_i} \to \bigoplus_{i < j}\bQ_{p,E_{ij}} \to 0
    \end{align}
    given by \autoref{lem:exact_dual_complex}. Then the long exact sequence in cohomology gives
    \begin{align}\label{eq:LES_Phi_K}
        0\to H^1_{\et}(E,\bQ_p)\to\bigoplus_{i} H^1_{\et}(E_i,\bQ_p) \overset{\Phi}{\longrightarrow}\bigoplus_{i<j}H^1_{\et}(E_{ij},\bQ_p)\to H^2_{\et}(E,\bQ_p)\to 0
    \end{align}
    (the sequence in $H^0$ groups is exact because the graph is a path). As the two outer groups vanish by \autoref{threefold_Q_p_rationality}, the map $\Phi$ is an isomorphism. On the other hand, from the above computations, we have \[ \dim_{\bQ_p}\left(\bigoplus_{i<j}H^1_{\et}(E_{ij},\bQ_p)\right) = n - 1, \] and \[ \dim_{\bQ_p}\left(\bigoplus_{i} H^1_{\et}(E_i,\bQ_p)\right) = n-2, \] so we have a contradiction.\medskip
        
    \textbf{Case 4: the dual complex is a two--dimensional connected closed compact manifold.} We are in case \autoref{prop:dlt_blowup_gorenstein_dual_complex_it_rc} of \autoref{prop:dlt_blowup_gorenstein_dual_complex}, so each $T_i$ is birational to $\bP^2$. In particular, we have $H^1(T_i, \cO_{T_i}) = H^2(T_i, \cO_{T_i} ) = 0$ so for $L \in \{ \bF_p, \bQ_p \}$, we have \[ H^r_{\et}(E_i, L) \expl{=}{\autoref{lem:dlt_blowup_gorenstein_comps_of_exc}} H^r_{\et}(T_i, L) \expl{=}{\autoref{we_suck_at_writing}} 0 \] for all $r>0$. Similarly, as $D_i$ is universally homeomorphic to $(E-E_i)|_{E_i}$, every component of an intersection $E_{ij}$ for $i<j$ is universally homeomorphic to a smooth rational curve. Since $E_{ij}$ is the disjoint union of its irreducible components by \autoref{prop:dlt_structure}.\autoref{prop:dlt_struct_geometry}, we deduce as above that \[H^r_{\et}(E_{ij}, L)=0\] for all $r>0$ and $L\in\{\bF_p,\bQ_p\}$. 
    
    Hence, as the higher $\bQ_p$-- and $\bF_p$--cohomology of all strata of $E$ vanishes, we obtain by \autoref{lem:cohomology_dual_complex_and_coherent} that for all $r \geq 0$ and $L\in\{\bF_p,\bQ_p\}$, we have 
    \[H^r(\cD(E), L) = H^r_{\et}(E, L).\]
    Now, take $L = \bQ_p$. We then deduce that \[\chi(\cD(E), \bQ_p)=\chi_{\et}(E,\bQ_p) \expl{=}{\autoref{threefold_Q_p_rationality}} 1,\] so by the classification of two--dimensional closed compact manifolds, $\cD(E)$ must be the real projective plane $\bR\bP^2$. Now, take $L = \bF_p$. We then obtain in particular that \[ 0 \expl{=}{$p \neq 2$} H^2(\bR\bP^2, \bF_p) = H^2(\cD(E), \bF_p) = H^2_{\et}(E, \bF_p), \] contradicting \autoref{eq:non_vanishing_H2}.
\end{proof}

\begin{remark}\label{example:singular_enriques_surfaces}

Already in the case where the dual complex is a point, we used that $p \neq 2$. This is necessary, as the following example shows: Let $S$ be a singular Enriques surface. As already mentioned in the introduction, a cone over $S$ has $\bQ_p$--rational singularities. It is also $F$--pure since $S$ is globally $F$--split (use \cite[Proposition 5.3]{Schwede_Smith_Globally_F_regular_an_log_Fano_varieties}). It is also quasi--Gorenstein and strictly log canonical, since $\omega_S \cong \cO_S$. This shows that \autoref{main_thm_threefolds} fails for $p = 2$.

In the cases where the dual complex was a path graph or a two--dimensional manifold, we also used that $p \neq 2$. It would be interesting to find counterexamples in those cases when $p = 2$.
\end{remark}


\section{The case of fourfolds}

\begin{thm}[{\autoref{Theorem_C}}]\label{main_thm_foufolds}
    Let $X$ be a normal fourfold of finite type over a perfect field $k$ of characteristic $p > 5$. Assume that $X$ is quasi--Gorenstein, $F$--pure, has $\bQ_p$--rational singularities and satisfies \autoref{hyp}. Then $X$ has canonical singularities.
\end{thm}
\begin{proof}
    As in the threefold case, we assume by contradiction that $X$ does not have canonical singularities. First of all, by \autoref{main_thm_threefolds}, we know that $X$ is canonical in codimension 3, so by shrinking we may assume that some closed point $x \in X$ is the only log canonical center. Let $\pi \colon Y \to X$ be a dlt modification (see \autoref{existence_dlt_modifications}), and let $E$ denote the exceptional divisor, so that $K_Y + E \sim \pi^*K_X$. We may now base change this setup to $\overline{k}$, and hence assume that $k$ is algebraically closed. As in the surface case, we then know by our hypotheses that $H^i_{\et}(E, \bQ_p) = 0$ for all $i > 0$. 

    Since $\wk \colon E^{\wk} \to E$ is a universal homeomorphism by \autoref{lem:weak_nor_univ_homeo}, we know by \autoref{lem:topological_invariance_of_etale_cohomology} that $H^i_{\et}(E^{\wk}, \bQ_p) = 0$ for all $i > 0$, and hence $\chi_{\et}(E^{\wk}, \bQ_p) = 1$. By \autoref{cor:full_exc_div_GFS_and_K_trivial}, we also know that $E^{\wn}$ is globally $F$--split and $\omega_{E^{\wn}} \cong \cO_{E^{\wn}}$.

    Assume for now that $E^{\wk}$ is Cohen--Macaulay. Then combining $\omega_{E^{\wk}} \cong \cO_{E^{\wk}}$, Serre duality and Cohen--Macaulayness of $E^{\wk}$, we deduce that \[ \chi(E^{\wk}, \cO_{E^{\wk}}) = 0 \] (recall that $\dim(E^{\wk}) = 3$, which is an odd number). On the other hand, we have \[ \chi(E^{\wk}, \cO_{E^{\wk}}) \expl{=}{\autoref{lem:baby_Riemann_Hilbert}}  \chi_{\et}(E^{\wk}, \bF_p) \expl{=}{\autoref{lem:euler_char}} \chi_{\et}(E^{\wk}, \bQ_p) = 1,  \] giving us our final contradiction. \medskip

    Thus, we are left to show that $E^{\wk}$ is Cohen--Macaulay. Let $E = \bigcup_i E_i$ denote the decomposition of $E$ into irreducible components. Since each $E_i$ is normal up to universal homeomorphism (see \autoref{prop:dlt_structure}.\autoref{prop:dlt_struct_geometry}), we know by \autoref{lem:topological_invariance_of_etale_cohomology} that 
    \begin{equation}\label{eq:end_paper}
        \cO_{E_i}^{1/p^{\infty}} = \nu_*\cO_{E_i^{\nu}}^{1/p^{\infty}},
    \end{equation} where $\nu \colon E_i^{\nu} \to E_i$ denotes the normalization morphism. Since each $E_i^{\nu}$ is a dlt type threefold by \autoref{prop:dlt_structure}.\autoref{prop:dlt_structure_adjunction} and $p > 5$, it is Cohen--Macaulay by \cite[Theorems 3 and 19] {Bernasconi_Kollar_Vanishing_theorems_for_threefolds_in_char_p}. By \cite[Tags ~\href{http://stacks.math.columbia.edu/tag/0955}{0955} and ~\href{http://stacks.math.columbia.edu/tag/0AVZ}{0AVZ}]{stacks-project}, this is equivalent to saying that for all closed points $x \in E_i^{\nu}$ and $j < \dim(X)$, we have the vanishing \[ H^j_x(E_i^{\nu}, \cO_{E_i^\nu}) = 0, \] where $H^j_x(X, -)$ denotes the $j$'th derived functor of the functor $H^0_x(X, -)$ of global sections supported at the point $x$ on abelian sheaves. Since this is a purely topological notion and since this functor commutes with colimits, this shows that the same vanishing holds for $\cO_{E_i^{\nu}}^{1/p^\infty}$. Given that $\nu$ is a homeomorphism, \autoref{eq:end_paper} gives us that the same vanishing holds for the sheaf $\cO_{E_i}^{1/p^\infty}$ on $E_i$, namely $\cO_{E_i}^{1/p^\infty}$ is Cohen--Macaulay.
    
    Note that for $l \geq 2$, the sheaves $\cO_{E_{i_1 \dots i_l}}^{1/p^{\infty}}$ are also Cohen--Macaulay. Indeed, this follows from the fact that $\dim(E_{i_1 \dots i_l}) \leq 2$, that $(E_{i_1 \dots i_l})^{\wn}$ is $S_2$ (see \autoref{prop:dlt_structure}.\autoref{prop:dlt_struct_geometry}) and the same argument as above. By \autoref{lem:exact_dual_complex} we have an exact sequence \[ 0 \to \bF_{p, E} \to \bigoplus_i \bF_{p, E_i} \to \bigoplus_{i < j}\bF_{p, E_{ij}} \to \cdots. \]
    By the Riemann--Hilbert correspondence (\cite[Theorem 10.2.7]{Bhatt_Lurie_A_Riemann-Hilbert_correspondence_in_positive_characteristic}), we deduce an exact sequence \begin{equation}\label{pf:les_exceptional_perfection}
        0 \to \cO_E^{1/p^{\infty}} \to \bigoplus_{i} \cO_{E_i}^{1/p^{\infty}} \to \bigoplus_{i < j} \cO_{E_{ij}}^{1/p^{\infty}} \to \cdots.
    \end{equation} 
    Let us now show that $\cO_E^{1/p^{\infty}}$ is Cohen--Macaulay. Let $x \in E$ be a closed point. By \autoref{pf:les_exceptional_perfection}, we know that for all $l \geq 0$, \[ H^l_x\left(E, \cO_E^{1/p^{\infty}}\right) \cong H^l_x\left(E, \: \bigoplus_{i} \cO_{E_i}^{1/p^{\infty}} \to \bigoplus_{i < j} \cO_{E_{ij}}^{1/p^{\infty}} \to \cdots \right), \]
    where the right term denotes hyper--local--cohomology. In particular, we have a spectral sequence \[ H^a_x\left(E, \bigoplus_{i_1, \dots, i_{b + 1}}\cO_{E_{i_1\dots i_{b + 1}}}^{1/p^{\infty}}\right) \implies H^{a + b}_x\left(E, \cO_E^{1/p^{\infty}}\right). \] For $a + b < 3$, the left term above must vanish, since $\cO_{E_{i_1\dots i_{b + 1}}}^{1/p^{\infty}}$ is Cohen--Macaulay of dimension $3 - b$ (see \autoref{prop:dlt_structure}.\autoref{prop:dlt_struct_geometry}). A straightforward spectral sequence argument then shows that $\cO_E^{1/p^{\infty}}$ is Cohen--Macaulay. Given that $E^{\wk} \to E$ is a universal homeomorphism (\autoref{lem:weak_nor_univ_homeo}), we can argue as previously to obtain that $\cO_{E^{\wk}}^{1/p^{\infty}}$ is also Cohen--Macaulay. By global $F$--splitness of $E^{\wk}$, the sheaf $\cO_{E^{\wk}}$ is a direct summand of $\cO_{E^{\wk}}^{1/p^{\infty}}$, whence $E^{\wk}$ is indeed Cohen--Macaulay.   
\end{proof}

\bibliographystyle{amsalpha} 
\bibliography{includeNice}
\end{document}